\documentclass[12pt]{amsart}
\usepackage{etex}

\usepackage[utf8]{inputenc} 
\usepackage{amsmath}
\usepackage{amsfonts}
\usepackage{amscd}
\usepackage{amssymb,amsthm}
\usepackage{enumerate}
\usepackage{tikz,tikz-cd}
\usepackage{mathtools}
\usepackage{mathrsfs}
\usetikzlibrary{calc}
\usepackage{hyperref}
\usepackage{bbm}
\usepackage{thmtools} 
\usepackage{tikz,tikz-cd}
\usetikzlibrary{shapes,calc,arrows,intersections,decorations.pathreplacing,decorations.markings,shapes.geometric,calc,positioning,backgrounds}
\usepackage{enumitem}

\declaretheorem[name=Theorem, numberwithin=section]{theorem}
\declaretheorem[name=Lemma,sibling=theorem]{lemma}

\newtheorem{corollary}[theorem]{Corollary}
\newtheorem{proposition}[theorem]{Proposition}

\newtheorem{definition}[theorem]{Definition}

\newtheorem{example}[theorem]{Example}

\usepackage{cleveref} 

\newcommand{\tr}{\operatorname{tr}} 
\newcommand{\E}{\mathbb{E}}          
\newcommand{\Prob}{\mathbb{P}}
\newcommand{\Hilb}{\mathcal{H}} 
\newcommand{\Var}{\mathrm{Var}}  

\newcommand{\CC}{\mathbb{C}}
\newcommand{\A}{\mathcal{A}}
\newcommand{\N}{\mathbb{N}}
\newcommand{\I}{\mathcal{I}}
\newcommand{\pp}{\mathcal{P}}
\newcommand{\s}{\mathcal{S}}
\newcommand{\K}{\mathcal{K}}
\newcommand{\kk}{{\bf \bar{k}}}

\newcommand{\gp}{\mathop{\begin{tikzpicture}[baseline]
\node[circle, fill=cyan, draw=black, inner sep=0pt, minimum size=3pt] (mid) at (0, 2.5pt) {};
\foreach \x in {0,60,...,300} {
\node[circle, fill=cyan, draw, inner sep=0pt, minimum size=3pt] (\x) at ($(mid)!6pt!(mid.\x)$) {};
\draw (\x) -- (mid) ;
}
\end{tikzpicture}}}

\begin{document}

\title[SYK-Free]{ Limit joint distributions of SYK models with partial interactions, mixed q-Gaussian models and asymptotic $\varepsilon$-freeness}
\author{Weihua Liu}
\address{School of Mathematics, Zhejiang University, Hangzhou, Zhejiang 310058, China}
\email{\url{lwh.math@zju.edu.cn}}

\author{Haoqi Shen}
\address{School of Mathematics, Zhejiang University, Hangzhou, Zhejiang 310058, China}
\email{\url{haoqi_shen@zju.edu.cn}}

\maketitle

\begin{abstract}
 We study the joint distribution of SYK Hamiltonians for different systems with specified overlaps. 
 We show that, in the large-system limit, their joint distribution converges in distribution to a mixed $q$-Gaussian system.
 We  explain that the graph product of diffusive abelian von Neumann algebras is isomorphic to a $W^*$-probability space generated by the corresponding $\varepsilon$-freely independent random variables with semicircular laws which form a special case of mixed $q$-Gaussian systems that can be approximated by our SYK Hamiltonian models.
Thus, we obtain a random model for asymptotic $\varepsilon$-freeness.

\end{abstract}

\section{Introduction}

In 1993, Sachdev and Ye introduced a new random model involving Gaussian distributions and spin operators to study quantum spin glasses and non-Fermi liquids \cite{SY93}. 
Later,  in 2015, the Sachdev--Ye--Kitaev (SYK) model, a simple variant of the Sachdev--Ye model  for  $n$ Majorana fermions with random interactions,  was proposed by Kitaev in his talk \cite{Ki15}. 
The general form of the SYK Hamiltonian is given by
$$H_{\mathrm{SYK}}=(\sqrt{-1})^{q_n / 2} \frac{1}{\sqrt{\frac{n^{q_n-1}}{\left(q_n-1\right)!J^2}}} \sum_{1 \leq i_1<i_2<\cdots<i_{q_n} \leq n} J_{i_1 i_2 \cdots i_{q_n}} \psi_{i_1} \psi_{i_2} \cdots \psi_{i_{q_n}},$$
where the $J_{i_1 i_2 \cdots i_{q_n}}$'s are independent standard Gaussian variables and the $ \psi_i$'s are fermionic operators satisfying the canonical anticommutation relations 
$$\psi_i \psi_j+\psi_j \psi_i=2 \delta_{i,j}.$$

The SYK model with large-$N$ Majorana fermions and fixed interaction orders $q_n = 2, 4$ was studied in \cite{MS16}. 
Subsequently, a more general case of mathematical interest, in which the interaction length $q_n$ depends on $n$, was investigated by Feng, Tian, and Wei~\cite{FTW19}.
In their setting,  the coefficients $(\sqrt{-1})^{q_n / 2} /\sqrt{\frac{n^{q_n-1}}{\left(q_n-1\right)!J^2}}$ of the model 
are replaced by $(\sqrt{-1})^{\left\lfloor q_n / 2\right\rfloor}/{\binom{n}{q_n}^{1 / 2}}$, ensuring that $H_{\mathrm{SYK}}$ is self-adjoint with mean zero and variance one. 
Namely, they considered the following Hamiltonian:
\begin{equation}\label{model: FTW}
H=(\sqrt{-1})^{\lfloor q_n / 2 \rfloor} \frac{1}{\sqrt{\binom{n}{q_n}}} \sum_{1 \leq i_1<i_2<\cdots<i_{q_n} \leq n} J_{i_1 i_2 \cdots i_{q_n}} \psi_{i_1} \psi_{i_2} \cdots \psi_{i_{q_n}}.
\end{equation}

The fermion operators $\psi_i$ can be realized using Pauli matrices:

$$
\sigma_1=\left(\begin{array}{ll}
0 & 1 \\
1 & 0
\end{array}\right), \quad \sigma_2=\left(\begin{array}{cc}
0 & -i \\
i & 0
\end{array}\right), \quad \sigma_3=\left(\begin{array}{cc}
1 & 0 \\
0 & -1
\end{array}\right).
$$

For $n=2r$, each Majorana fermion is constructed as an $r$-fold tensor product:
$$
\begin{array}{rlrl}
\psi_1 & =\sigma_1 \otimes I_2 \otimes \cdots \otimes I_2 & \psi_{r+1} & =\sigma_2 \otimes I_2 \otimes \cdots \otimes I_2 \\
\psi_2 & =\sigma_3 \otimes \sigma_1 \otimes \cdots \otimes I_2 & \psi_{r+2} & =\sigma_3 \otimes \sigma_2 \otimes \cdots \otimes I_2 \\
\vdots & & \vdots & \\
\psi_r & =\sigma_3 \otimes \sigma_3 \otimes \cdots \otimes \sigma_1 & \psi_{2 r} & =\sigma_3 \otimes \sigma_3 \otimes \cdots \otimes \sigma_2
\end{array}
$$
where the $I_2$ in the tensor products represents the $2 \times 2$ identity matrix. 
For $n=2r-1$, we simply select $2r-1$ of these elements.
In general, by the GNS representation of the algebra generated by $\psi_i$'s with respect to the normalized trace $\operatorname{tr}(A)=\frac{1}{2^{n/2}} \operatorname{Tr}(A)$, the $\psi$'s can be realized on a $2^{n/2}$-dimensional complex Hilbert space.

Before the work of Feng, Tian, and Wei on SYK models, a related model, the quantum $q$-spin glass model, was considered in~\cite{ES14,KLW15a,KLW15b}. 
For $q \geq 1$, the Hamiltonian of a quantum $q$-spin glass is
\begin{equation}\label{model: ES}
H_n^q=3^{-q / 2}\binom{n}{q}^{-1 / 2} \sum_{1 \leq i_1<\cdots<i_q \leq n} \sum_{a_1, \ldots, a_q=1}^3 \alpha_{a_1, \ldots, a_q,\left(i_1, \ldots, i_q\right)} \sigma_{i_1}^{a_1} \cdots \sigma_{i_q}^{a_q},
\end{equation}
where the coefficients $\alpha_{a_1, \ldots, a_q, (i_1, \ldots, i_q)}$ play the role of $J_{i_1 i_2 \cdots i_{q_n}} $ and   are i.i.d. random variables with mean $0$ and variance $1$ and $\sigma_i^a$ are defined as 
$$
\sigma_i^a=I_2^{\otimes(i-1)} \otimes \sigma_a \otimes I_2^{\otimes(n-i)},
$$
where $\sigma_a$ (for $a=1,2,3$) are the three Pauli matrices introduced earlier.

The main difference between \eqref{model: FTW} and \eqref{model: ES} is that the latter considers all possible combinations of Pauli matrices, whereas the former uses only a subset of them in order to realize the CAR relations.
Erd\H{o}s and Schr\"oder showed in~\cite{ES14} that the limiting density of $H_n^q$ has a phase transition in the regimes $q^2 \ll n$, $q^2 / n \rightarrow a$ and $q^2 \gg n$.
In~\cite{FTW19}, the parity of $q_n$ plays an important role.
They demonstrated that the large-$N$ distribution of the empirical measure of eigenvalues of the SYK model depends on the limit of $q^2_n / n$ and on the parity of $q_n$.
As a result, if $q_n$ is a constant, the SYK model becomes a very sparse random matrix model, exhibiting behavior similar to classical Brownian motion or the Bernoulli distribution.  
When $q^2_n / n \to \lambda \neq 0$, it was shown using the moment method that the limiting distribution is the $q$-Gaussian distribution, where
 $q=(-1)^{q_n}e^{-2\lambda}$ when $q_n$ are of the same parity. 
 
In summary, the Erd\H{o}s--Schr\"oder model yields the $q$-Gaussian distribution for $0 \le q \le 1$, while the Feng--Tian--Wei model yields the $q$-Gaussian distribution for $-1 \le q \le 1$.
The properties of this distribution were studied via $q$-Hermite orthogonal polynomials~\cite{Sz26}, the study of which dates back to Rogers in 1894~\cite{Ro93}.
On the other hand, the $q$-Gaussian distribution naturally arises from a deformation model interpolating between CCR and CAR, which is related to free probability theory.

Free probability was introduced by Voiculescu to address the isomorphism problem of free group von Neumann algebras. 
In this theory, a noncommutative analogue of the classical independence relation, called free independence, was introduced. 
The study of free independence is deeply connected to the reduced free product of probability spaces, which arises as a modification of the Fock space construction.
More importantly, the free central limit law, namely the semicircular law, is the spectral distribution of the Gaussian operator, which is the sum of the creation and annihilation operators on a Fock space with the canonical inner product, with respect to the vacuum state.
In the CCR, CAR, and free Fock space settings with the canonical inner product, the relation between the creation operators $l_i$ and the annihilation operators $l_i^*$ is given by $l_i^* l_j - q l_j l_i^* = \delta_{i,j} $, where $q = 1, -1, 0$, corresponding to CCR, CAR, and the free case, respectively.
By introducing a \lq\lq twisted"  inner product on the algebraic Fock space of a Hilbert space $\Hilb$ with orthonormal basis $\{e_i\mid i \in I\}$, Bo\.zejko and Speicher constructed examples of creation operators $l(e_i)$ and annihilation operators $ l(e_i)^* $ such that $$l(e_i)^* l(e_j) - q l(e_j) l(e_i)^* = \delta_{i,j} $$ for all $-1\leq q \leq 1$~\cite{BS91}.  
For fixed $q$, a $q$-Gaussian operator  is defined to be $l(e_i)+l(e_i)^*$. 
By computing the orthogonal polynomials~\cite{Sp97},  the spectral distribution of $q$-Gaussian operators with respect to the vacuum states $\tau_\Hilb$ is exactly the $q$-Gaussian distribution.
Following the Brownian motion-like properties of $q$-Gaussian distributions, Pluma and Speicher extended the SYK model to the multivariate case and showed that the limiting joint distribution of independent SYK models of the same form is exactly the joint distribution of a $q$-Gaussian system~\cite{PS22}. 
Therefore, SYK models provide a random model for $q$-Gaussian variables.
Some other random matrix models for $q$-Gaussian operators were introduced by Speicher~\cite{Sp92} and \'Sniady~\cite{Sn01}.

In the context of the random SYK model, where models with different interaction lengths lead to various limit laws, a natural question arises:

{\bf Question:} What is the limiting joint distribution of independent SYK models of different forms?

It is evident that we need to determine the relations between $q$-Gaussian operators. 
In fact, we will see that these operators satisfy a mixed $q$-relation, which generalizes the $q$-relation introduced by Speicher~\cite{Sp93}. In the mixed $q$-Gaussian setting, the creation and annihilation operators satisfy the relation
$$    l(e_i)^* l(e_j) - q_{i,j} l(e_j) l(e_i)^* = \delta_{i,j},$$
where $-1\leq q_{i,j} \leq 1$ and $q_{i,j} = q_{j,i}$. We call the family $\{s_i=l(e_i)+l(e_i)^* \mid i\in I\}$ a $Q=(q_{i,j})$-system. 
The existence of mixed $q$-Gaussian systems was first established through probabilistic methods~\cite{Sp93} and, more recently, using ultraproduct techniques by Junge and Zeng~\cite{JZ24}. More importantly, the mixed $q$-Gaussian system can be constructed using "twisted" inner products on Fock spaces~\cite{BS94}, which naturally leads to questions in operator algebra theory.
 
To avoid confusion, we will use $r_n$ to denote the interaction length of  SYK models.  Then, we have the following result.

\begin{theorem}
Let $\I$ be an index set.  For each $k\in \I$ and $n\in \N$, let $H_{k,n}$ be the SYK model such that 
$$H_{k,n}=\frac{(\sqrt{-1})^{\left\lfloor r_{k,n} / 2\right\rfloor}}{\binom{n}{r_{k,n}}^{1 / 2}} \sum_{1 \leq i_1<\cdots<i_{r_{k,n}} \leq n} J_{k;i_1, \ldots, i_{r_{k,n}}} \psi_{i_1} \cdots \psi_{i_{r_{k,n}}}$$
where the $J_{k;i_1,\dots,i_{r_{k,n}}}$ are independent
 random variables with $\mathbb{E}[J_{k;i_1,\dots,i_{r_{k,n}}}]=0$
and $\mathbb{E}[J_{k;i_1,\dots,i_{r_{k,n}}}^2]=1$, and such that
for each integer $p\ge 3$ there exists a constant $C_p<\infty$ such that
$$
|\mathbb{E}\bigl[J_{k;i_1,\dots,i_{r_{k,n}}}^p\bigr]|\le C_p,
$$
and that the $\psi_{i}$ are Majorana fermions satisfying the canonical anticommutation relations. 
For each $k\in \I$, assume that the $r_{k,n}$'s all share the same parity for all $n$ and 
$$
\lim\limits_{n\rightarrow \infty}\frac{r_{i,n}}{n}=0, \quad \lim\limits_{n\rightarrow \infty}\frac{r_{i,n}r_{j,n}}{n} = \lambda_{i,j} \in[0, \infty].
$$
Then, for $k_1,...,k_d\in \I$, we have 
$$
\lim\limits_{n \rightarrow \infty} \mathbb{E}\left[\operatorname{tr} \left( H_{k_1,n} \cdots H_{k_d,n} \right)\right]
=\tau_\Hilb\left(s_{k_1}\cdots s_{k_d}\right),
$$
where $\{s_i\mid i\in \I\}$ is a $Q=(q_{i,j})$-system such that  $q_{i,j}=(-1)^{r_{i,n}r_{j,n}}e^{-2\lambda_{i,j}}.$
\end{theorem}
The condition that, for $k_1,...,k_d\in \I$,
$$
\lim\limits_{n \rightarrow \infty} \mathbb{E}\left[\operatorname{tr} \left( H_{k_1,n} \cdots H_{k_d,n} \right)\right]
=\tau_\Hilb\left(s_{k_1}\cdots s_{k_d}\right),
$$
where $\{s_i\mid i\in \I\}$ denotes the mixed $q$-Gaussian system,  means that the family of SYK models converges in distribution to the mixed $q$-Gaussian system. We will write
$$
\{H_{k,n}\mid k\in \I\} \xrightarrow[]{d}\{s_k\mid k\in \I \}
$$
for simplicity.

The above theorem shows that once the limiting distribution of the SYK models is determined, then the $Q$-relation is also determined.  
This naturally raises a question: how can we realize other values for $q_{i,j}$? 
To answer this question, we find it natural to consider interacting SYK models from different systems with certain overlaps.
This also allows us to study how subsystems of SYK models affect each other when the systems satisfy certain specified relations.
Indeed, we have the following result.
\begin{theorem}\label{main theorem}
Let $\I$ be an index set.  
For each $k\in \I$ and $n\in \N$, let $A_{k,n}\subset \N$ be such that $|A_{k,n}|=n$ and let $r_{k,n}\in \N$ satisfy $\lim\limits_{n\rightarrow \infty}\frac{r_{k,n}}{n}=0$ and assume that  the $r_{k,n}$'s all share the same parity for all $n$. 
Let
$$ 
  H_{k,n}=\frac{(\sqrt{-1})^{\left\lfloor r_{k,n} / 2\right\rfloor}}{\binom{n}{r_{k,n}}^{1 / 2}} \sum_{ i_1<\cdots<i_{r_{k,n}} \in A_{k,n}} J_{k;i_1, \ldots, i_{r_{k,n}}} \psi_{i_1} \cdots \psi_{i_{r_{k,n}}},
$$
where the $J_{k;i_1,\dots,i_{r_{k,n}}}$ are independent
 random variables with $\mathbb{E}[J_{k;i_1,\dots,i_{r_{k,n}}}]=0$
and $\mathbb{E}[J_{k;i_1,\dots,i_{r_{k,n}}}^2]=1$, and such that
for each integer $p\ge 3$ there exists a constant $C_p<\infty$ such that
$$
|\mathbb{E}\bigl[J_{k;i_1,\dots,i_{r_{k,n}}}^p\bigr]|\le C_p,
$$
and that the $\psi_{i}$ are Majorana fermions satisfying the canonical anticommutation relations. 
If
$$
\lim\limits_{n\rightarrow \infty}\frac{r_{i,n}r_{j,n}}{n} \cdot \frac{\lvert A_{i,n}\cap A_{j,n} \rvert}{n}=\lambda_{i,j} \in[0, \infty],
$$
then 
$$
\{H_{k,n}\mid k\in \I\} \xrightarrow[]{d}\{s_k\mid k\in \I \}
$$
where $\{s_k\mid k\in \I\}$ is a $Q=(q_{i,j})$ system such that  $q_{i,j}=(-1)^{r_{i,n}r_{j,n}}e^{-2\lambda_{i,j}}.$
\end{theorem}
In the Feng--Tian--Wei model, $r_{k,n}$ is assumed to satisfy $r_{k,n}\leq n/2$, which allows for the possibility that $\lim\limits_{n\rightarrow \infty}\frac{r_{k,n}}{n}\neq0$. 
They restrict to  $r_{k,n} \le n/2$because replacing $r_{k,n}$ by $n-r_{k,n}$ yields an essentially equivalent model. 
However, the case $\lim_{n\to\infty} \frac{r_{k,n}}{n} \neq 0$ is not covered by our result. 
Indeed, we provide an example for which $\lim_{n\to\infty} \frac{r_{k,n}}{n} \neq 0$ for certain $k$, and the limit in \Cref{main theorem} fails.
The example is given at the end of Section 4.

Based on the above theorem, in addition to the relations we studied for SYK models with intersections from different systems, we can now construct more interesting mixed $q$-Gaussian systems.

Compare with the random models in \cite{JZ24,Sn01,Sp92,Sp93}, which are realized as random operators, the (mixed) $q$-Gaussian relations are implemented by placing a suitable probability measure on the canonical anticommutation relation (CAR) elements. 

The SYK models are closer in spirit to Gaussian ensembles, which are related to other interesting topics; for example, the strong convergence for the random matrices of the form
$X=\sum_{i=1}^n g_i A_i,$
where $A_i\in M_d(\mathbb{C})_{\mathrm{sa}}$ are deterministic $d\times d$ self-adjoint matrices and $g_i$ are i.i.d.  Gaussian random variables, were studied by Bandeira, Boedihardjo, and van Handel in \cite{BBV23}.

In noncommutative probability, free independence represents the \lq\lq highest level of noncommutativity", whereas classical independence is entirely commutative. 
The above theorem illustrates a relationship between noncommutativity and the interaction length in SYK models and quantum systems. 
Specifically, a longer interaction length in the SYK model or greater overlap among different quantum systems corresponds to increased noncommutativity. By adjusting the overlap between quantum systems, it becomes possible to realize $\varepsilon$-free relations between free semicircular elements, thereby obtaining an asymptotic $\varepsilon$-free result analogous to that in free probability~\cite{Vo91}.
Building on the bridge between free group operator algebras and random matrices established by Voiculescu, powerful tools such as free entropy were developed and have led to solutions of many longstanding problems \cite{Vo93,Vo94,Vo95,Ra92,Dy93,Ha22,HJK25}.
Our work may have further applications in operator algebras related to mixed $q$-Gaussian systems and $\varepsilon$-free group algebras.

$\varepsilon$-free independence, introduced by Młotkowski, can be viewed as a mixture of classical and free products of probability spaces. 
While free independence is realized by reduced free products and classical independence by tensor product, $\varepsilon$-free independence is realized through a universal construction---namely, graph products. 
In fact, graph products preceded the concept of $\varepsilon$-independence; they were first introduced by Green in her thesis~\cite{Gr90}.   
The concept was later applied to operator algebras by Caspers and Fima~\cite{CF17}. 
Subsequently, several random matrix models for $\varepsilon$-independence (Graph products) were introduced in~\cite{CC21, CDHJEN25, ML19}.
At the end of Morampudi and Laumann's work, they posed an open problem regarding the SYK model and $\varepsilon$-freeness~\cite{ML19}. In our work, we provide a solution to their question.

Besides this introductory section, the rest of paper is organized as follows:

In Section 2, we provide more details about the mixed $q$-Gaussian system and related properties.

In Section 3, we study $\varepsilon$-free independence.

In Section 4, we provide more background on the SYK model and outline a rough strategy for the proof of \Cref{main theorem}, including the key lemmas. Since the proofs are technical, we collect them in Section 5. At the end of Section 5, we offer some comments and further questions arising from our work.

\section{Mixed $q$-Gaussian System}
In this section, we will assume that $\Hilb$ is a Hilbert space with orthonormal basis $\{e_i\mid i\in \I\}$, where $\I$ is an index set. 
The  algebraic full Fock space of $\Hilb$ is $\mathcal{F}_{\text{alg}}(\Hilb) = \bigoplus_{n\geq0}\Hilb^{\otimes n}$, where $\Hilb^{\otimes 0} := \mathbb{C}\Omega$ for a unit vector $\Omega$ called the vacuum. 
Given a symmetric matrix $Q=(q_{i,j})_{i,j \in \I}$ such that $q_{i,j}=q_{j,i}$ and $|q_{i,j}|\leq 1$, 
one can define a deformed pre-inner product on $\mathcal{F}_{\text{alg}}(\Hilb)$ denoted $\langle \cdot, \cdot \rangle_Q$, such that 
$ \langle e_{i_1} \otimes \cdots \otimes e_{i_n}, \Omega\rangle_Q =0$ for $n\geq 1$ and 
$$
\langle e_{i_1} \otimes \cdots \otimes e_{i_m}, e_{j_1} \otimes \cdots \otimes e_{j_n}\rangle_Q = \sum_{k=1}^n \delta_{i_1, j_k} q_{i_1,j_1} \cdots q_{i_1,j_{k-1}} \langle e_{i_2}\otimes \dots \otimes e_{i_m}, e_{j_1} \otimes \cdots \otimes e_{j_{k-1}} \otimes e_{j_{k+1}} \otimes \cdots \otimes e_{j_n}\rangle_Q. 
$$
This recursive relation determines the inner product of vectors with the same tensor length, while vectors with different tensor lengths are orthogonal.  
The explicit formulas of the inner product for vectors of the same tensor length are not used in this paper and can be found in~\cite{JZ24,Lu99}. The positivity of $\langle \cdot, \cdot \rangle_Q$ was proved in~\cite{BS94}, where a more general family of inner products related to the Yang--Baxter relation was introduced. 
In addition, if $q_{i,j}>0$ for all $i,j\in \I$, then the above pre-inner product is indeed an inner product.

For each $i\in \I$,  define the left creation operator $l_i$ on  $\mathcal{F}_{\text{alg}}(\Hilb)$ by
$$ l_i\Omega = e_i, \quad l_i(e_{j_1} \otimes \cdots \otimes e_{j_n}) = e_i \otimes e_{j_1} \otimes \dots \otimes e_{j_n},  $$
and the adjoint $l_i^*$ of $l_i$, which is called the left annihilation operator satisfies 
$$ l_i^*\Omega = 0, \quad l_i^*(e_{j_1} \otimes \dots \otimes e_{j_n}) = \sum_{k=1}^n \delta_{i, j_k} q_{i ,j_1} \cdots q_{i, j_{k-1}} e_{j_1} \otimes \cdots \otimes e_{j_{k-1}} \otimes e_{j_{k+1}} \otimes \cdots \otimes e_{j_n}.$$
In this setting, the creation operators and annihilation operators satisfy the following mixed $q$-commutation relation:
$$ l_i^* l_j - q_{i,j} l_j l_i^* = \delta_{i,j}.$$

For each $i\in \I$, the $q$-Gaussian operator $s_i$ is then defined to be $l_i+l_i^*$. The family $\{s_i \mid i\in \I\}$ forms a $q$-Gaussian system with respect to the vacuum state $\tau_\Hilb(\cdot)=\langle \cdot \Omega, \Omega\rangle_Q$. Actually, we are interested in the value of  the moments $\langle s_{i_1}\cdots s_{i_m} \Omega, \Omega\rangle_Q$.

In the following, we present a formula for computing joint moments of a given mixed $q$-Gaussian system, which is used in this paper.  
To achieve this, we provide an explicit formula for $ s_{\epsilon(1)}\cdots s_{\epsilon(d)}\Omega$.
Before proceeding, we first introduce several notations and definitions.

\begin{definition}
Let $k\in \N$. We denote by $[k]=\{1,\dots,k\}$, the elements of $[k]$ are of natural order.
\begin{enumerate}
\item A  partition $\pi$ of a set $\s$ is a collection of disjoint subsets of $\s$ such that their union equals $\s$. The elements of $\pi$ are called blocks. The family of partitions of $\s$ is denoted by $\pp(\s)$.  
\item A partition whose blocks all contain exactly $2$ elements is called a pair partition. The family of pair partitions of a set $\s$ is denoted by $\pp_2(\s)$.
\item For partitions $\pi,\sigma\in \pp(\s)$, we say $\pi \leq \sigma$ if every blocks of $\pi $ is contained in a block of $\sigma$.
\item Given a sequence $(i_1,...,i_n)\in \s^n$ which defines a map $\epsilon:[n]\rightarrow \s$ such that $\epsilon(k)=i_k$, we define $\ker \epsilon=\{\epsilon^{-1}(s)\neq \emptyset \mid s\in \s\}$. 
\end{enumerate}
\end{definition}

For an ordered set $\s$,  blocks of pair partitions of $\s$ are written in order, e.g. $v=(e,z)$ with $e<z$.  
Consequently, given a pair partition $\pi\in \pp_2(\s)$, there is  a natural order on its blocks defined by $(e_1,z_1)<(e_2,z_2)$ if and only if $e_1<e_2$.  
Two pair blocks $(e_1,z_1)<(e_2,z_2)$ are said to be crossing if $e_1<e_2<z_1<z_2$.
For a fixed $d$, let $V\subset [d]$ such that $|V|=2m\leq d$ , and let 
$$\pp_{12}([d],V)=\{ \pi\cup \{ \{w\}\mid w\in V^c \} \mid \pi \in \pp_2(V) \}$$ 
be the family of partitions whose blocks consist of  pair blocks  from $V$ and singletons  from $V^c=[d]\setminus V$.  
Let $\epsilon$ be a map from $[d]$ to $\I$.  
Notice that $s_{\epsilon(1)} \cdots s_{\epsilon(d)} $ is a polynomial in annihilation and  creation operators.
The key idea in computing $s_{\epsilon(1)} \cdots s_{\epsilon(d)} \Omega $ is to move the annihilation operators next to their corresponding creation operators  by the mixed $q$-relation, leaving the remaining creation operators to act on $\Omega$. 
After identifying the elimination pairs, we record how to twist the corresponding annihilation and creation operators by the following quantities: Let $\sigma\in \pp_{12}([d],V)$. For $i,j\in \I$, we define 
 $$ C_{1}(\sigma, \epsilon; i, j) := \#\left\{ (v_1, v_2) \in \sigma \times \sigma \mid v_1 \subset \epsilon^{-1}(i), v_2\subset \epsilon^{-1}(j), v_1<v_2 \text{ and } v_1, v_2 \text{ are crossing} \right\}$$
and     
        $$C_{2}(\sigma,\epsilon; i, j) := \#\left\{ (w, v) \in V^{c} \times \sigma \mid w \in \epsilon^{-1}(i), v=(e,z) \subset \epsilon^{-1}(j), e < w < z \right\}.$$
The quantity $C_{1}(\sigma,\epsilon; i, j) $ counts the number of $q_{i,j}$ terms arising from pair blocks, while $C_{2}(\sigma,\epsilon; i, j) $ counts the number of $q_{i,j} $ terms involving a pair block and a singleton.  
Thus, the total number of  $q_{i,j}$ terms obtained for a partition $\sigma\in \pp_{12}([d],V)$ is defined to be 
   $$ C(\sigma,\epsilon; i, j) := C_{1}(\sigma,\epsilon; i, j) + C_2(\sigma,\epsilon; i, j).$$  
For   $W = [d]\setminus V = \{ w_1 < \cdots < w_{d-2m}\}$, we denote $e^{\otimes W}=e_{\epsilon(w_1)}\otimes\cdots \otimes e_{\epsilon( w_{d-2m})}$. 

\begin{proposition}\label{prop:s Omega}
Let $\epsilon:[d]\rightarrow \I$ and let $\{s_i\mid i\in \I\}$ be the $Q=(q_{i,j})$-Gaussian system. Then, we have   
    $$s_{\epsilon(1)}\cdots s_{\epsilon(d)}\Omega=\sum\limits_{m=0}^{\lfloor d/2\rfloor }\sum\limits_{\substack{V\subset[d]\\ |V|=2m}}\sum\limits_{\substack{\sigma\in \pp_{12}([d],V)\\ \sigma\leq \ker \epsilon}}\prod\limits_{i,j\in \I} q_{i,j}^{C(\sigma,\epsilon; i, j)} e^{\otimes[d]\setminus V}, $$
\end{proposition}

\begin{proof}
We prove the proposition by induction on $d$. 
It is obviously true for $d=1$. We assume the formula holds for $d$. For the inductive step, we consider the action of $s_{\epsilon(0)}$ on the expansion for $d$ (assuming for convenience that the new element is $0 < x$ for all $x \in [d]$). Set $X=\{0\}\cup [d]$, $W=[d]\setminus V=\{w_{1}, \cdots, w_{d-2m}\}$, $k_0=\epsilon(0)$ for convenience. By induction, we have 
$$s_{\epsilon(0)}\cdots s_{\epsilon(d)}\Omega
= \left(l_{\epsilon(0)}+l^*_{\epsilon(0)}\right)\sum\limits_{m=0}^{\lfloor d/2\rfloor }\sum\limits_{\substack{V\subset[d]\\ |V|=2m}}\sum\limits_{\substack{\sigma\in \pp_{12}([d],V)\\ \sigma\leq \ker \epsilon}}\prod\limits_{i,j\in \I} q_{i,j}^{C(\sigma,\epsilon; i, j)} e^{\otimes[d]\setminus V}.$$
To the first part of the sum we have 
\begin{align*}l_{\epsilon(0)}\sum\limits_{m=0}^{\lfloor d/2\rfloor }\sum\limits_{\substack{V\subset[d]\\ |V|=2m}}\sum\limits_{\substack{\sigma\in \pp_{12}([d],V)\\ \sigma\leq \ker \epsilon}}\prod\limits_{i,j\in \I} q_{i,j}^{C(\sigma,\epsilon; i, j)} e^{\otimes[d]\setminus V}
&=\sum\limits_{m=0}^{\lfloor d/2\rfloor }\sum\limits_{\substack{V\subset[d]\\ |V|=2m}}\sum\limits_{\substack{\sigma\in \pp_{12}([d],V)\\ \sigma\leq \ker \epsilon}}\prod\limits_{i,j\in \I} q_{i,j}^{C(\sigma,\epsilon; i, j)} e_{\epsilon(0)}\otimes e^{\otimes[d]\setminus V}\\
&=\sum\limits_{m=0}^{\lfloor d/2\rfloor }\sum\limits_{\substack{V\subset[d]\\ |V|=2m}}\sum\limits_{\substack{\sigma\in \pp_{12}(X,V)\\ \sigma\leq \ker \epsilon}}\prod\limits_{i,j\in \I} q_{i,j}^{C(\sigma,\epsilon; i, j)}  e^{\otimes X\setminus V}\\
&=\sum\limits_{m=0}^{\lfloor d/2\rfloor }\sum\limits_{\substack{V\subset X, 0\not\in V\\ |V|=2m}}\sum\limits_{\substack{\sigma\in \pp_{12}(X,V)\\ \sigma\leq \ker \epsilon}}\prod\limits_{i,j\in \I} q_{i,j}^{C(\sigma,\epsilon; i, j)}  e^{\otimes X\setminus V}.
\end{align*}

Compared to  the desired expression, the missing terms involve partitions from $\pp_{12}(X)$ where a pair block contains $0$. 
Let us now turn to the second part of the sum.
For a fixed $1\leq r\leq d-2m$, let 
    $$C_3(W,\epsilon; i,w_r)=\#\{w| w<w_r,\epsilon(w)=i\}.$$
Then, for fixed $V$ with $|V|=2m$ and $\sigma\in \pp_{12}([d],V)$, we have 
\begin{align*}
l^*_{\epsilon(0)}\prod\limits_{i,j\in \I} q_{i,j}^{C(\sigma,\epsilon; i, j)} e^{\otimes[d]\setminus V} 
&= \prod\limits_{i,j\in\I} q_{i,j}^{C(\sigma,\epsilon; i, j)}                   \sum_{r=1}^{d-2m} \delta_{k_0,\epsilon(w_r)}  \prod_{s\in \I} q_{k_0,s}^{C_3(W,\epsilon;s, w_r)}  e^{\otimes (W \setminus \{w_r\})}\\
&= \sum_{r=1}^{d-2m} \delta_{k_0,\epsilon(w_r)}  \prod_{i,j\in \I} q_{i,j}^{C(\sigma,\epsilon; i, j)+\delta_{k_0,i}\cdot C_3(W,\epsilon; j,w_r)}  e^{\otimes (W \setminus \{w_r\})}. \\ 
\end{align*}

We now need to compare the coefficients of the vector  $e^{\otimes (W \setminus \{w_r\})}$ for a fixed $V$ and $\sigma\in \pp_{12}([d],V)$.  
Notice that  the map $\sigma\rightarrow \sigma'=(\sigma\setminus\{w_r\})\cup \{(0,w_r)\}$ is a bijection from $\pp_{12}([d],V)$ to the subset of $\pp_{12}(X,V\cup\{0,w_r\})$ consisting of elements that contain the pair block $(0,w_r)$.  
For $k_0 \neq i $, the contribution of the block $(0, w_r)$ to the quantity $C(\sigma',\epsilon; i, j)$ is zero, so we have $C(\sigma',\epsilon; i, j)=C(\sigma,\epsilon; i, j)$.
If $k_0=i$, then we have 

\begin{align*}
      C_{1}(\sigma',\epsilon; i, j)&=C_{1}(\sigma,\epsilon; i, j)+ \#\{v= (e,z)\in \sigma| \epsilon(e)=\epsilon(z)=j, \text{and}\, (0,w_r),v\, \text{are crossing}\}\\
     &=C_{1}(\sigma,\epsilon; i, j)+ \#\{v=(e,z)\in\sigma| e<w_r<z,\epsilon(e)=\epsilon(z)=j\}
\end{align*}
and 
\begin{align*}
      C_{2}(\sigma',\epsilon; i, j)=C_{2}(\sigma,\epsilon; i, j)&- \#\{v=(e,z)\in\sigma| e<w_r<z,\epsilon(e)=\epsilon(z)=j\}\\
      &+\#\{w\in W |0<w<w_r,\epsilon(w)=i\}. 
\end{align*}
We have the first negative term because the singleton $\{w_r\}$ is now  part of a pair. The second positive term arises from the contribution of the pair $(0, w_r)$ and is equal to $C_3(W,\epsilon;i, w_r)$. It follows that
\begin{align*}
C(\sigma',\epsilon; i, j)&=C_{1}(\sigma',\epsilon; i, j)+C_{2}(\sigma',\epsilon; i, j)=C_{1}(\sigma,\epsilon; i, j)+C_2(\sigma,\epsilon; i, j)+C_3(W,\epsilon; i,w_r)\\
&=C(\sigma,\epsilon; i, j)+C_3(W,\epsilon; i,w_r).
\end{align*}

In summary, for  a fixed $V\subset [d]$ such that $|V|=2m$, we have that  

\begin{align*}
&l^*_{\epsilon(0)}\sum\limits_{\substack{\sigma\in \pp_{12}([d],V)\\ \sigma\leq \ker \epsilon}}\prod\limits_{i,j\in\I} q_{i,j}^{C(\sigma,\epsilon; i, j)} e^{\otimes[d]\setminus V}\\
&=\sum\limits_{\substack{\sigma\in \pp_{12}([d],V)\\ \sigma\leq \ker \epsilon}}\sum_{r=1}^{d-2m} \delta_{k_0,\epsilon(w_r)}  \prod_{i,j\in \I} q_{i,j}^{C(\sigma,\epsilon; i, j)+\delta_{k_0,i}\cdot C_3(W,\epsilon; j,w_r)}  e^{\otimes (W \setminus \{w_r\})}\\ 
&=\sum_{r=1}^{d-2m} \sum\limits_{\substack{\sigma'\in \pp_{12}(X,V\cup\{0,w_r\})\\ \sigma'\leq \ker \epsilon, (0,w_r)\in \sigma'}} \prod_{i,j\in \I} q_{i,j}^{C(\sigma',\epsilon; i, j)}  e^{\otimes (X\setminus (V\cup \{0,w_r\}))}\\
&=\sum\limits_{\substack{\sigma'\in \pp_{12}(X,V'), 0\in V'\\|V'|=2m+2\\ \sigma'\leq \ker \epsilon}} \prod_{i,j\in \I} q_{i,j}^{C(\sigma', \epsilon; i, j)}  e^{\otimes (X\setminus V')}. 
\end{align*}

Thus, we get desired coefficient for  all the terms involving partitions from $\mathcal{P}_{12}(X)$ where a pair block contains zero. 
This completes the proof by induction.
\end{proof}

By letting $V=[d]$ and $cr(\pi,\epsilon; i,j)=C_1(\pi,\epsilon; i,j)$ for $\pi\in \pp_2(d)$, we get the following formula for moments of  mixed $q$-Gaussian operators.

\begin{proposition}\label{mixed q-Gaussian moments}Let $\epsilon:[d]\rightarrow \I$ and let $\{s_i\mid i\in \I\}$ be the $Q=(q_{i,j})$-Gaussian system.  Then, we have 
$$
\tau_\Hilb(s_{\epsilon(1)}\cdots s_{\epsilon(d)})=\sum\limits_{\substack{\pi\in \pp_2(d)\\\pi\leq\ker \epsilon}} \prod\limits_{i,j\in \I} q^{cr(\pi,\epsilon; i,j)}_{i,j}.$$
\end{proposition}
When $d$ is an odd number, we always have $\tau_\Hilb(s_{\epsilon(1)}\cdots s_{\epsilon(d)})=0$.

In the following, we derive a characterization of $q$-Gaussian systems arising from random matrix models. 
A similar result for mixed $q$-Gaussian systems appears in Section 4.

For $(d_n)_{n\in \N}\subset \N$ such that $\lim\limits_{n\rightarrow \infty} d_n=\infty$,  assume that $\K_n$ is a finite set with $\lim\limits_{n\rightarrow \infty} |\K_n|=\infty$ and suppose that we have a sequence $(e^{(n)}_{k})_{k\in \K_n}\subset M_{d_n}(\CC)$ such that 
\begin{itemize}
\item $(e^{(n)}_{k})_{k\in \K_n}$ are selfadjoint.
\item $(e^{(n)}_{k})_{k\in \K_n}$ are orthonormal with respect to $tr_{d_n}(\cdot)=\frac{1}{d_n}Tr_{d_n}(\cdot)$, namely $ tr_{d_n}(e^{(n)}_{k}e^{(n)}_{j})=\delta_{k,j}$.
\item $\|e^{(n)}_{k}\|\leq 1$ for all $k\in \K_n$.
\end{itemize}
Then, for $m\geq 2$ we have that 
$$
\begin{aligned}
|tr_{d_n}( e^{(n)}_{k_1}\cdots e^{(n)}_{k_m})|^2 &\leq |tr_{d_n}( e^{(n)}_{k_1}\cdots e^{(n)}_{k_{m-1}}e^{(n)}_{k_{m-1}}\cdots e^{(n)}_{k_{1}})||tr_{d_n}( e^{(n)}_{k_m}e^{(n)}_{k_m})|\\
&\leq |tr_{d_n}( e^{(n)}_{k_1}e^{(n)}_{k_{1}})||tr_{d_n}( e^{(n)}_{k_m}e^{(n)}_{k_m})|\\
&\leq 1.
\end{aligned}
$$
where the first inequality follows from the Cauchy-Schwarz inequality and the second inequality holds since $\|e^{(n)}_{k}\|\leq 1$ for all $k\in \K_n$.

Let $\{J_{i,k,n}\mid i\in \I, k\in \K_n\}$ be identically distributed random variables with mean zero, variance one and finite $l$-th moments $M_l$ for all $l\in \N$.
Let $A_{i}^{(n)}=\frac{1}{\sqrt{|\K_n|}}\sum\limits_{k\in \K_n} J_{i,k,n}e^{(n)}_{k}$. 
Then, for each $n\in\N$, $\{ A_{i}^{(n)}|i\in\I\}$ is a family of  independent random matrices.
Assume that for all $m\in \N$ and  $\epsilon:[m]\rightarrow \I$, the following limit exists
$$\lim\limits_{n\rightarrow \infty} \E\left[tr_{d_n}( A_{\epsilon(1)}^{(n)}\cdots A_{\epsilon(m)}^{(n)})\right].$$
Then
$$\E \left[tr_{d_n}(A_{i}^{(n)}) \right]=0,\quad \E \left[tr_{d_n}\left(\left(A_{i}^{(n)}\right)^2 \right)\right]=1.$$
Thus, given $m\in \N$ and  $\epsilon:[m]\rightarrow \I$, we have that 
$$
\begin{aligned}
\E\left[tr_{d_n}\left( A_{\epsilon(1)}^{(n)}\cdots A_{\epsilon(m)}^{(n)}\right)\right]=\frac{1}{|\K_n|^{m/2}}\sum\limits_{k_1,...,k_m\in\K_n}\E[J_{\epsilon(1),k_1,n}\cdots J_{\epsilon(m),k_m,n} ] tr_{d_n}(e_{k_1}^{(n)}\cdots e_{k_m}^{(n)}).
\end{aligned}
$$

Let  $\kk=(k_1,..,k_m) $. Then 
$$
\begin{aligned}
\sum\limits_{\kk\in\K_n^{m}}=\sum\limits_{\ker \kk \in P_2(m)}+\sum\limits_{\ker \kk \not\in P_2(m)}.
\end{aligned}
$$

For each tuple $\kk$ such that $\ker \kk \notin \pp_2(m)$, if $\kk$ has a singleton, then $$\E[J_{\epsilon(1),k_1,n}\cdots J_{\epsilon(m),k_m,n} ] tr_{d_n}(e_{k_1}^{(n)}\cdots e_{k_m}^{(n)})=0.$$ 
Therefore, the nonzero terms in the second summation have the property that $\ker \kk$ contains a block with at least three elements, and all the other blocks have sizes greater than or equal to $2$.
Let $$a_m=\max\{ M_{l_1}\cdots M_{l_\iota} \mid l_1+\cdots+l_\iota=m, l_1,...l_\iota \geq 2\}$$ and 
$$\pp_{2,3}(m)=\{\pi\in \pp (m)\mid |V|\geq 2,\, \forall V\in \pi;\, \exists V\in \pi\, s.t.\, |V|\geq 3 \}.$$
Then 
 $$
\begin{aligned}
&\left|\sum\limits_{\ker \kk \not\in \pp_2(m)}\E[J_{\epsilon(1),k_1,n}\cdots J_{\epsilon(m),k_m,n} ] tr_{d_n}(e_{k_1}^{(n)}\cdots e_{k_m}^{(n)})\right|\\
\leq& \sum\limits_{\ker \kk \in \pp_{2,3}(m)}|\E[J_{\epsilon(1),k_1,n}\cdots J_{\epsilon(m),k_m,n} ] tr_{d_n}(e_{k_1}^{(n)}\cdots e_{k_m}^{(n)})|
\leq& \sum\limits_{\ker \kk \in \pp_{2,3}(m)} a_m.
\end{aligned}
$$

For each $\pi\in \pp_{2,3}(m)$, we have $$\#\{\kk\in \K_n^m|\ker\kk=\pi\}\leq |\K_n|^{\frac{m-1}{2}}.$$ 
Therefore, we have that 
$$\left|\sum\limits_{\ker \kk \not\in \pp_2(m)}\E[J_{\epsilon(1),k_1,n}\cdots J_{\epsilon(m),k_m,n} ] tr_{d_n}(e_{k_1}^{(n)}\cdots e_{k_m}^{(n)})\right|\leq  a_m|\pp_{2,3}(m)| |\K_n|^{\frac{m-1}{2}}.$$
It follows that 
$$\lim\limits_{n\rightarrow \infty} \frac{1}{|\K_n|^{m/2}}\sum\limits_{\ker \kk \not\in P_2(m)}\E[J_{\epsilon(1),k_1,n}\cdots J_{\epsilon(m),k_m,n} ] tr_{d_n}(e_{k_1}^{(n)}\cdots e_{k_m}^{(n)})=0.$$

Thus, we have 
$$
\begin{aligned}
\lim\limits_{n\rightarrow \infty}\E\left[tr_{d_n}\left( A_{\epsilon(1)}^{(n)}\cdots A_{\epsilon(m)}^{(n)}\right)\right]=\lim\limits_{n\rightarrow \infty}\frac{1}{|\K_n|^{m/2}}\sum\limits_{\ker \kk\in P_2(m)}\E[J_{\epsilon(1),k_1,n}\cdots J_{\epsilon(m),k_m,n} ] tr_{d_n}(e_{k_1}^{(n)}\cdots e_{k_m}^{(n)}).
\end{aligned}
$$

Notice that, for $\kk=(k_1,...,k_m)$ such that $\ker \kk$ is a pair partition,  $\E[J_{\epsilon(1),k_1,n}\cdots J_{\epsilon(m),k_m,n} ]\neq 0$ only if $\ker\kk\leq \ker \epsilon$. Therefore, we have that

\begin{equation}\label{limiting distribution}
\lim\limits_{n\rightarrow \infty}\E\left[ tr_{d_n} \left( A_{\epsilon(1)}^{(n)}\cdots A_{\epsilon(m)}^{(n)}\right)\right] =\lim\limits_{n\rightarrow \infty}\sum\limits_{\substack{\ker \kk\in P_2(m)\\ \ker\kk\leq \ker \epsilon}} \frac{1}{|\K_n|^{m/2}}\E[J_{\epsilon(1),k_1,n}\cdots J_{\epsilon(m),k_m,n} ] tr_{d_n}(e_{k_1}^{(n)}\cdots e_{k_m}^{(n)}).
\end{equation}

The above limit is nonzero only if $m$ is an even number. 
For this model, we have the following proposition.
\begin{proposition} Let $\I$ be an infinite index set. 
Assume that for all $d\in \N$ and $\epsilon:[2d]\to \I$ such that $\pi=\ker \epsilon \in \pp_2(2d)$, we have that 
$$\lim\limits_{n\rightarrow \infty}\E\left[ tr_{d_n}\left( A_{\epsilon(1)}^{(n)}\cdots A_{\epsilon(2d)}^{(n)}\right)\right]= q^{cr(\pi)}$$
Then, $\{A_{i}^{(n)}\mid i\in \I\}$ converges to the $q$-Gaussian system $\{s_i\mid i\in\I\}$ in distribution.
\end{proposition}
\begin{proof}
According to the definition of $q$-Gaussian systems, it  suffices to show that 
$$\lim\limits_{n\rightarrow \infty}\E\left[ tr_{d_n}\left( A_{\epsilon(1)}^{(n)}\cdots A_{\epsilon(m)}^{(n)}\right)\right]=\sum\limits_{\substack{\pi\in \pp_2(m), \pi\leq \ker \epsilon}}  q^{cr(\pi)},$$
for all $m\in \N$ and $\epsilon: [m]\rightarrow \I$.

By \eqref{limiting distribution},  the above statement is clearly true for odd $m$.  When $m$ is even, we have  
$$
\begin{aligned}
&\lim\limits_{n\rightarrow \infty}\E\left[ tr_{d_n} \left( A_{\epsilon(1)}^{(n)}\cdots A_{\epsilon(m)}^{(n)}\right)\right]\\
&=\lim\limits_{n\rightarrow \infty}\sum\limits_{\substack{\ker \kk\in P_2(m)\\ \ker\kk\leq \ker \epsilon}} \frac{1}{|\K_n|^{m/2}}\E[J_{\epsilon(1),k_1,n}\cdots J_{\epsilon(m),k_m,n} ] tr_{d_n}(e_{k_1}^{(n)}\cdots e_{k_m}^{(n)})\\
&=\lim\limits_{n\rightarrow \infty}\sum\limits_{\substack{\pi \in P_2(m)\\ \pi\leq \ker \epsilon}} \sum\limits_{\ker \kk=\pi} \frac{1}{|\K_n|^{m/2}}\E[J_{\epsilon(1),k_1,n}\cdots J_{\epsilon(m),k_m,n} ] tr_{d_n}(e_{k_1}^{(n)}\cdots e_{k_m}^{(n)})\\
&=\sum\limits_{\substack{\pi \in P_2(m)\\ \pi\leq \ker \epsilon}} \lim\limits_{n\rightarrow \infty}\sum\limits_{\ker \kk=\pi} \frac{1}{|\K_n|^{m/2}}\E[J_{\epsilon(1),k_1,n}\cdots J_{\epsilon(m),k_m,n} ] tr_{d_n}(e_{k_1}^{(n)}\cdots e_{k_m}^{(n)}).
\end{aligned}
$$

On the other hand, for each $\pi\in \pp_2(m)$ such that $\pi\leq \ker\epsilon$, we can define a map $\epsilon_\pi:[m]\rightarrow \I$ such that $\ker \epsilon_\pi=\pi$.  
Then we have 
$$\E[J_{\epsilon(1),k_1,n}\cdots J_{\epsilon(m),k_m,n} ] =\E[J_{\epsilon_\pi(1),k_1,n}\cdots J_{\epsilon_\pi(m),k_m,n} ] $$

By assumption, we have 
$$
\begin{aligned}
&\lim\limits_{n\rightarrow \infty}\sum\limits_{\ker \kk=\pi} \frac{1}{|\K_n|^{m/2}}\E[J_{\epsilon_\pi(1),k_1,n}\cdots J_{\epsilon_\pi(m),k_m,n} ] tr_{d_n}(e_{k_1}^{(n)}\cdots e_{k_m}^{(n)})\\
=& \lim\limits_{n\rightarrow \infty}\E\left[ tr_{d_n}\left( A_{\epsilon_\pi(1)}^{(n)}\cdots A_{\epsilon_\pi(m)}^{(n)}\right)\right]=q^{cr(\pi)}.
\end{aligned}
$$

The statement follows.
\end{proof}

Building on the previous proposition, to show that $\{A_i^{(n)} \mid i \in \I\}$ converges in distribution to the $q$-Gaussian system $\{s_i \mid i \in \I\}$, it suffices to verify the case of pair partitions, which is typically addressed in the first step for this kind of problem. 
As a result, Theorem 3.1 in~\cite{PS22} follows straightforwardly from Equation (20) in~\cite{FTW19}. Indeed, to establish a limiting mixed $q$-Gaussian distribution for the SYK model, it suffices to verify the limiting moments corresponding to pair partitions.

\section{$\varepsilon$-free independence}

In this section, we introduce the notion of $\varepsilon$-free independence and show its relation to mixed $q$-Gaussian systems.

Recall that a noncommutative probability space, $(\A, \phi)$, consists of a unital  algebra $\A$ over a field $\mathbb{K}$ and  a unital linear functional $\phi: \A \rightarrow \mathbb{K}$  such that $\phi\left(1_\A\right)=1$.
Elements of $\A$ are called random variables.
According to~\cite{Sp97}, there are exactly two unital, universal, symmetric independence relations for collections of random variables from $\A$, applied to the subalgebras they generate: classical independence and free independence.
Free independence requires a noncommutative framework; in a commutative setting, all random variables are constant except at most one.
A family of subalgebras $\A_i$ of $\A$ is said to be free if 
 $$\phi(a_1\cdots a_n)=0 $$
 whenever $a_k\in \A_{i_k}$, $i_1\neq i_2\cdots \neq i_n$ and $\phi(a_k)=0$ for all $k$.
 Readers are referred to~\cite{NS, VDN92} for background and more details on free probability.
Classical independence is defined for commutative random variables, and has a definition analogous to that of freeness.
A family of subalgebras $\A_i$ of $\A$, such that $\A_i, \A_j$ commute for all $i\neq j$, is said to be classically independent if 
 $$\phi(a_1\cdots a_n)=0 $$
 whenever $a_k\in \A_{i_k}$, $i_k$ are pairwise distinct and $\phi(a_k)=0$ for all $k$.  We do not require $\A_i$ to be a commutative algebra. 
 
M\l otkowski~\cite{Ml04} introduced $\varepsilon$-free independence as a generalization of classical and free products of probability spaces.
This concept has been further developed in~\cite{EPS18,SW16}. 

Let $\I$ be a finite or infinite index set, and let $\varepsilon_{i,j}=\varepsilon_{j,i}\in\{0,1\}$ for $i,j\in\I$.
The subset $\I_\varepsilon^n \subset \I^{n}$  is defined to consist of $n$-tuples $\mathbf{i}=(i(1), \ldots, i(n))$ satisfying: if $i(k)=i(l)$ for $1 \leq k < l \leq n$, then there exists a $k<p<l$ such that $i(k) \neq i(p)$ and $\varepsilon_{i(p),i(k)}=0$. 
In the $\varepsilon$-independence relation, the diagonal elements do not play an important role at this moment.
\begin{definition}\label{epsilon independence}
Let $(\A, \phi)$ be a noncommutative probability space. The subalgebras $\A_i \subset \A, i \in \I$, are $\varepsilon$-free independent if and only if
\begin{itemize}
\item[(i)] the algebras $\A_i$ and $\A_j$ commute for all $i \neq j$ for which $\varepsilon_{i,j}=1$,
\item[(ii)] for $(i(1), \ldots, i(n)) \in \I_\varepsilon^n$ and $\left(a_1, \ldots, a_n\right) \in \A_{i(1)} \times \cdots \times \A_{i(n)}$, such that $\phi\left(a_k\right)=0$ for all $1 \leq k \leq n$, it follows that $\phi\left(a_1 \cdots a_n\right)=0$.
\end{itemize}
\end{definition}

\subsection{Mixed q-Gaussian system and $\varepsilon$-free independence}

In the  following, we assume that  $l_i^{*}l_j - q_{i,j} l_j l_i^* = \delta_{i,j} \cdot \mathbf{1}$ for $i,j\in \I$.
\begin{lemma} \label{q_ij=1 implies commuting relation} For $i\neq j$, if $q_{i,j}=1$, then $l_il_j=l_jl_i$.
\end{lemma}
\begin{proof}
It suffices to show that $e_i\otimes e_j=e_j\otimes e_i$.  
Let $v=e_{i_1}\otimes \cdots \otimes e_{i_n}$ for $i_1,\cdots, i_n\in \I$.
If $n\neq 2$, we have $\langle e_i\otimes e_j,v\rangle=\langle e_j\otimes e_i,v\rangle=0$
For $n=2$, if $(i_1,i_2) \neq (i,j), (j,i) $, we also have $\langle e_i\otimes e_j,v\rangle=\langle e_j\otimes e_i,v\rangle=0$.
Therefore, we only need to show  that $\langle e_i\otimes e_j,e_i\otimes e_j\rangle=\langle e_j\otimes e_i,e_i\otimes e_j\rangle$, which is equivalent to $\langle l_i^*l_j^*l_il_j\Omega,\Omega\rangle=1$. Notice that $q_{i,j}=1$, we have $l_il^*_j=l_j^*l_i$ and $l_jl^*_i=l_i^*l_j$. The statement follows.
\end{proof} 

\begin{proposition}
Assume that $q_{i,j}\in\{0,1\}$ for $i\neq j$ and $q_{i,i}\in [-1,1]$, and let $\A$ be the unital *-algebra generated by $\{ l_i\mid i\in \I\}$, $\tau_\Hilb$ be the vacuum state on $\A$ and $\A_i$ be the unital *-algebra generated by $l_i$. Let $\varepsilon_{i,j}=q_{i,j}$. Then, $\A_i$'s are $\varepsilon$-free in $(\A,\tau_\Hilb)$.
\end{proposition}
\begin{proof}
(1)  For $i\neq j$ such that $q_{i,j}=1$,  by definition and \Cref{q_ij=1 implies commuting relation}, we have 
$[l_i,l_j]=0$ and $[l_i, l_j^*]=0$.  It follows that $\A_i$ and $\A_j$ commute.

(2) Let $(i(1), \ldots, i(d)) \in \I_\varepsilon^{d}$ and $\left(a_1, \ldots, a_d\right) \in \A_{i(1)} \times \cdots \times \A_{i(d)}$, such that $\tau_\Hilb\left(a_k\right)=0$ for all $1 \leq k \leq d$. 
Applying the commutation relation  $ l_i^{*}l_i- q_{i,i} l_i l_i^* =\mathbf{1}$, each $a_k\in \A_{i(k)}$ can be written as 
$$a_k=\sum_{n, m \geq 0} \alpha_{n, m}^{(k)} l_{i(k)}^n l_{i(k)}^{* m} \in \mathcal{A}_{i(k)} \quad(k=1, \ldots, d),$$
with finitely many nonzero $\alpha^{(k)}_{n,m}$.\\
Since $\tau_\Hilb\left(a_k\right)=0$, we have $\alpha_{0, 0}^{(k)}=0$.
Let $v_k=a_k \Omega=\sum\limits_{n\geq 1} \alpha^{(k)}_{n,0} e_{i(k)}^{\otimes n}$. 
We now show that $a_1 \cdots a_d\Omega=v_1\otimes \cdots \otimes v_d$, by induction.
$d=1$ is obvious.  Assume that  the statement is true for $d-1$,  then  we have
$$a_2 \cdots a_d\Omega=v_2\otimes \cdots \otimes v_d.$$
Notice that 
$$v_2\otimes \cdots \otimes v_d=\sum_{n_2, \ldots, n_d>0} \alpha_{n_2,0}^{(2)} \ldots \alpha_{n_d, 0}^{(d)}
e^{\otimes n_2}_{i(2)}\otimes \cdots \otimes e^{\otimes n_d}_{i(d)}.$$
If $i(1)\not\in \{i(2),...,i(d)\}$, then $l^*_{i(1)} v_2\otimes \cdots \otimes v_d=0$.\\
If $i(1)\in \{i(2),...,i(d)\}$, say $i(1)=i(j_1)$, according to the definition of $\I_\varepsilon^d$, there exists   $1<j_2<j_1$ such that $q_{i(1),i(j_2)}=\varepsilon_{i(1),i(j_2)}=0$. Apply the annihilation relation, we also have $l^*_{i(1)} v_2\otimes \cdots \otimes v_d=0$.
Therefore, 
$$a_1v_2\otimes \cdots \otimes v_d=\sum_{n\geq 1} \alpha_{n, 0}^{(1)} l_{i(1)}^n v_2\otimes \cdots \otimes v_d=v_1\otimes \cdots \otimes v_d.$$
It follows that $\tau_\Hilb\left(a_1 \cdots a_d\right)=0$.  The proof is complete.

\end{proof}

\begin{corollary}
Let $(s_1,...,s_n)$ be a mixed q-Gaussian system with  $q_{i,j}\in\{0,1\}$ for $i\neq j$, and let $\A$ be the unital algebra generated by $s_1,...,s_n$.  Then, $s_i$'s are $\varepsilon$-free in $(\A,\tau_\Hilb)$.
\end{corollary}

Given a family of  groups $\{G_i\mid i \in \I\}$  and $\varepsilon=(\varepsilon_{i,j})$, recall that the $\varepsilon$-product (or graph product) of $\{G_i\mid i \in \I\}$ is the quotient of the free product group $\star_{i \in \I} G_i$ by the relations that $G_i$ and $G_j$ commute whenever $\varepsilon_{i,j}=1$. 
Given a group $G$, we denote by $L(G)$ the group von Neumann algebra of $G$. 
For $q \in (-1,1] $, the spectral distribution of the $q $-Gaussian operator with respect to $\tau_\Hilb$ is diffuse; 
One should be careful that when $q=1$, the $q $-Gaussian operator is unbounded and has Gaussian spectral distribution with respect to the vacuum state. Fortunately, these unbounded operators exhibit favorable properties, such as being self-adjoint and densely defined. To analyze the von Neumann algebra generated by elements involving unbounded operators with a faithful tracial state, one can apply the Cayley transform.
Consequently, the von Neumann algebra generated by a single $q$-Gaussian operator for $q \in (-1, 1] $ is isomorphic to $ L(\mathbb{Z}) $.
For $q = -1 $, the corresponding spectral distribution is Bernoulli; thus, the von Neumann algebra generated by a single $ q $-Gaussian operator is isomorphic to $L(\mathbb{Z}_2)$.
Given $\varepsilon$ and a family of ($W^*$-) probability spaces $(\A_i,\phi_i)$, as in free probability, we have a probability space $(\A,\phi)$ such that there exist embeddings $\iota:\A_i\rightarrow \A$ such that
\begin{enumerate}
\item For all $a\in \A_i$, we have $\phi_i(a)=\phi(\iota_i(a))$,
\item $\A_i$'s are $\varepsilon$-free in $(\A,\phi)$.
\end{enumerate}
The construction uses graph products of operator algebras by Caspers and Fima~\cite{CF17}. 
Readers are referred to~\cite{ CSHJEN242,CSHJEN241,CDD252,CDD251} for more details on the construction and the properties of the associated von Neumann algebras.
In particular, we have the following isomorphism result.
\begin{proposition}
   Let $\mathcal {G}$ be a finite simple graph with adjacency matrix $\varepsilon=(\varepsilon_{i,j})$. Then the von Neumann algebra  $\gp_{\varepsilon} (L^{\infty}[0,1],d\mu)$,  the group von Neumann algebra $L(*_\varepsilon \mathbb{Z})$, and the $Q=(q_{i,j})$-Gaussian von Neumann algebra are isomorphic where $q_{i,j}=\varepsilon_{i,j}$ for $i\neq j$ and $q_{i,i}\in (-1,1]$ for all $i$.
\end{proposition}

Finally, we apply \Cref{main theorem} to obtain, as a special case, an asymptotic $\varepsilon$-freeness result for mixed $q$-Gaussian variables.

Let $Q=(q_{i,j})_{i,j=1,...,d}$ such that $q_{i,j}=q_{j,i}$, $q_{i,j}\in\{0,1\}$ for $i\neq j$, and $q_{i,i}=0$ for all $i$.
Let $n=d^2m$, and let  $B_{i,j,n}=\{ (i-1)dm+(j-1)m+k\mid k=1,...,m \}$ for $i,j=1,...,d$. Then, the sets $B_{i,j,n}$'s are pairwise disjoint and $|B_{i,j,n}|=m=\frac{n}{d^2}$.
Let $A'_{i,n}=\bigcup\limits_{j:\, q_{i,j}=0} \bigl(B_{i,j,n}\cup B_{j,i,n}\bigr)$. 
Then $|A'_{i,n}|\leq (2d-1)m\leq d^2m=n$, and 
$$A'_{i,n}\cap A'_{j,n}=\begin{cases}
\emptyset, & \text{if}\, q_{i,j}=1 \\[0.6em]
B_{i,j,n}\cup B_{j,i,n}, & \text{if}\, q_{i,j}=0.
\end{cases}$$

For each $k=1,\dots,d$, take $A''_{k,n}\subset \{ kn+1,kn+2,...,(k+1)n\}$ such that $|A''_{k,n}|=n-(2d-1)m$, and let $A_{k,n}=A'_{k,n}\cup A''_{k,n}$. Then, for $i\neq j$, 
$$\lvert A_{i,n}\cap A_{j,n}\rvert =\begin{cases}
0, & \text{if}\, q_{i,j}=1, \\[0.6em]
2m, & \text{if}\, q_{i,j}=0.
\end{cases}$$

For each $k=1,\dots,d$ and $n\in\N$, we thus obtain a set $A_{k,n}\subset \N$ such that $|A_{k,n}|=n$.
Then
$$|A_{i,n}\cap A_{j,n}|=2(1-q_{i,j})m=\frac{2(1-q_{i,j})n}{d^2}.$$

Let $r_{k,n}=2[\frac{n^{\frac{2}{3}}}{4}]$. Then $r_{k,n}$'s are even, and $\lim\limits_{n\rightarrow \infty}\frac{r_{k,n}}{n}=0.$
Moreover,
$$
\lim\limits_{n\rightarrow \infty}\frac{r_{i,n}r_{j,n}}{n} \cdot \frac{\lvert A_{i,n}\cap A_{j,n} \rvert}{n}=
\begin{cases}
0, & \text{if}\, q_{i,j}=1 \\[0.6em]
\infty, & \text{if}\, q_{i,j}=0.
\end{cases}
$$

Now, take 
$$ 
  H_{k,n}=\frac{(\sqrt{-1})^{\left\lfloor r_{k,n} / 2\right\rfloor}}{\binom{n}{r_{k,n}}^{1 / 2}} \sum_{ i_1<\cdots<i_{r_{k,n}} \in A_{k,n}} J_{k;i_1, \ldots, i_{r_{k,n}}} \psi_{i_1} \cdots \psi_{i_{r_{k,n}}},
$$
where the random variables $J_{k;i_1,\dots,i_{r_{k,n}}}$ are independent standard Gaussian random variables and $\psi_{i}$ are Majorana fermions satisfying the canonical anticommutation relations and can be chosen from an infinite-dimensional CAR algebra.  

Then, by \Cref{main theorem},
$\{H_{k,n}\mid k=1,\dots,d \}$ converges in distribution to the mixed $q$-Gaussian system $\{s_k\mid k=1,\dots,d \}$ with $q$-relation $Q$.

\section{SYK Models}

Recall that in \Cref{main theorem}, we are considering the following models 
$$
	H_{k,n}=\frac{(\sqrt{-1})^{\lfloor r_{k,n}/2\rfloor}}{\binom{n}{r_{k,n}}^{1/2}}\sum_{{i_1<\dots<i_{r_{k,n}}}\in A_{k,n}} J_{k;i_1,\dots,i_{r_{k,n}}}\,\psi_{i_1}\dots\psi_{i_{r_{k,n}}},
$$
where $\mathcal{I}$ is an index set, and for each $k\in\mathcal{I}$ , $A_{k,n} \subset \mathbb{N}$ with $|A_{k,n}| = n$;
the $J_{k;i_1,\dots,i_{r_{k,n}}}$'s are independent random variables 
and  the $\psi_{i}$'s are Majorana fermions satisfying the canonical anticommutation relations.  For simplicity, given an increasing sequence $R=(i_1,\dots,i_r)\subset \N$ with $ i_1<\cdots<i_r$,  we write
$$
J_{k,R}:=J_{k;i_1,\dots,i_r},
\quad
\Psi_R:=\psi_{i_1}\cdots\psi_{i_r}.
$$
Let  $
I_{k,n}=\bigl\{(i_1,\dots,i_{r_{k,n}})\in A_{k,n}^{\,r_{k,n}}|\ i_1<\cdots<i_{r_{k,n}}\bigr\}
$
and  $C_{k,n}:=\frac{(\sqrt{-1})^{\lfloor r_{k,n}/2\rfloor}}{\binom{n}{r_{k,n}}^{1/2}}$. The SYK Hamiltonians can be simply written as 
$$
H_{k,n}=C_{k,n}\sum_{R\in I_{k,n}} J_{k,R}\,\Psi_{R}.
$$

The following identities are standard (see~\cite{FTW19}) and will be used repeatedly:
\begin{itemize}
  \item[a.] For any increasing tuple $R$,
  \begin{equation}\label{eq:psi_selfadjoint}
  	 \tr\!\left((\sqrt{-1})^{\,2\lfloor |R|/2\rfloor}\,\Psi_R^{\,2}\right)=1.
  \end{equation}
  
  \item[b.] If $R \neq \emptyset$, then $\tr(\Psi_R) = 0$. Moreover, for any $R_1, \dots, R_d$, we have

  $$\bigl|\tr(\Psi_{R_1}\cdots\Psi_{R_d})\bigr|\in \{0,1\}.$$

  \item[c.] For any increasing tuples $A,B$,
  \begin{equation}\label{eq:psi_commute}
  	\Psi_A\Psi_B
  =
  (-1)^{|A||B|}(-1)^{|A\cap B|}\,\Psi_B\Psi_A.
  \end{equation}
\end{itemize}

In \Cref{main theorem}, the conditions on parameters are

\begin{itemize}
\item For any $i \in \mathcal{I}$, $\lim\limits_{n\to\infty} \frac{r_{i,n}}{n}=0$.
  \item For every $k\in\mathcal{I}$, $(r_{k,n})_{n\in \N}$ has the same parity, i.e., there exists $r_k\in\{0,1\}$ such that $$r_{k,n}\equiv r_k \pmod 2$$ for all $n$.

  \item For any $i,j\in\mathcal{I}$, the limit below exists:
$$
  \lim_{n\to\infty}\frac{r_{i,n}r_{j,n}}{n}\cdot\frac{|A_{i,n}\cap A_{j,n}|}{n}=\lambda_{i,j}\in[0,\infty].
 $$
 \item For any $i,j \in \mathcal{I}$, $q_{i,j}=(-1)^{r_i r_j}e^{-2\lambda_{i,j}}$.
\end{itemize}

To prove the theorem, we actually need to check the following limit
\begin{equation}\label{eq:EtrH}
\lim_{n\rightarrow \infty} 
\mathbb{E}\Bigl[\tr\bigl(H_{\epsilon(1),n} \cdots H_{\epsilon(d),n}\bigr)\Bigr]
=
\sum_{\substack{\pi \in \pp_2(d)\\ \pi \leq \ker \epsilon}} 
\prod_{i,j\in\mathcal I} q_{i,j}^{\,cr(\pi,\epsilon; i,j)}=\tau_\Hilb(s_{\epsilon(1)}\cdots s_{\epsilon(d)}).
\end{equation}
where $\epsilon:[d]\rightarrow \mathcal{I}$ is any index map.

Notice that 
$$\mathbb{E}\Bigl[\tr\bigl(H_{\epsilon(1),n}\cdots H_{\epsilon(d),n}\bigr)\Bigr]
=
\frac{(\sqrt{-1})^{\sum_{k=1}^{d}\lfloor r_{\epsilon(k),n}/2 \rfloor}}
{\prod_{k=1}^d\binom{n}{r_{\epsilon(k),n}}^{1/2}}
\sum_{R_{1}\in I_{\epsilon(1),n}}\cdots\sum_{R_{d}\in I_{\epsilon(d),n}}
\mathbb{E}\!\left[\prod_{k=1}^d J_{\epsilon(k),R_{k}}\right]\,
\tr\!\bigl(\Psi_{R_{1}}\cdots \Psi_{R_{d}}\bigr)  .$$

 Recall that $(\epsilon(1),R_{1}),...,(\epsilon(d),R_{d})$ are indices, we may consider the partition $$\ker[ (\epsilon(1),R_{1}),...,(\epsilon(d),R_{d})]\in \pp(d)$$ the family of partitions of $d$ elements.  
As in the case of the limiting distribution of eigenvalues for Wigner matrices, we will first show that
only the indices  $(\epsilon(1),R_{1}),...,(\epsilon(d),R_{d})$ such that  $\ker [(\epsilon(1),R_{1}),...,(\epsilon(d),R_{d})]\in \pp_2(d)$ contribute. 
\begin{lemma}
\label{lem:appear_twice}
 $$
\lim\limits_{n\rightarrow \infty}\frac{(\sqrt{-1})^{\sum_{k=1}^{d}\lfloor r_{\epsilon(k),n}/2 \rfloor}}
{\prod_{k=1}^d\binom{n}{r_{\epsilon(k),n}}^{1/2}}
\sum_{\substack{R_k \in I_{\epsilon(k),n} ,\forall k=1,...,d\\
\ker [(\epsilon(1),R_{1}),...,(\epsilon(d),R_{d})] \not\in \pp_2(d)}}
\mathbb{E}\!\left[\prod_{k=1}^d J_{\epsilon(k),R_{k}}\right]\,
\tr\!\bigl(\Psi_{R_{1}}\cdots \Psi_{R_{d}}\bigr)=0.
$$
\end{lemma}
\begin{proof}

Since $\E[ J_{\epsilon(k),R_{k}}]=0$, if $\ker [(\epsilon(1),R_{1}),...,(\epsilon(d),R_{d})]$ contains a block with one element, then 
$$\mathbb{E}\!\left[\prod_{k=1}^d J_{\epsilon(k),R_{k}}\right]=0.$$
Now, we assume that the sizes of blocks of  $\ker [(\epsilon(1),R_{1}),...,(\epsilon(d),R_{d})]$ are greater 1.
Let $\ker [(\epsilon(1),R_{1}),...,(\epsilon(d),R_{d})]=\pi=\{V_1,...,V_t\}.$
Then $|V_i|\geq 2$. For each $i=1,\dots,t$, choose a representative $v_i\in V_i$ and let $\beta_i= \epsilon(v_i)$.

We have
$$\#\left\{(R_1,\dots,R_d)\,\middle|\, R_k\in I_{\epsilon(k),n}\ \forall k=1,\dots,d,\ 
\ker[(\epsilon(1),R_1),\dots,(\epsilon(d),R_d)]=\pi \right\}\leq \prod\limits_{i=1}^t \binom{n}{r_{\beta_i,n}},$$
where
$$\prod_{k=1}^d\binom{n}{r_{\epsilon(k),n}}^{1/2}= \prod\limits_{i=1}^t \binom{n}{r_{\beta_i,n}}^{\frac{|V_i|}{2}}.
$$
Since $\mathbb{E}\!\left[\prod_{k=1}^d J_{\epsilon(k),R_{k}}\right]$ is uniformly bounded by a number $M$, and since$\lim\limits_{n\rightarrow \infty} \binom{n}{r_{\epsilon(k),n}}=\infty$, we have
$$
\left| \frac{(\sqrt{-1})^{\sum_{k=1}^{d}\lfloor r_{\epsilon(k),n}/2 \rfloor}}
{\prod_{k=1}^d\binom{n}{r_{\epsilon(k),n}}^{1/2}}
\sum_{\substack{R_k \in I_{\epsilon(k),n} ,\forall k=1,...,d\\
\ker [(\epsilon(1),R_{1}),...,(\epsilon(d),R_{d})]=\pi}}
\mathbb{E}\!\left[\prod_{k=1}^d J_{\epsilon(k),R_{k}}\right]\,
\tr\!\bigl(\Psi_{R_{1}}\cdots \Psi_{R_{d}}\bigr)\right|\leq \prod\limits_{i=1}^t \binom{n}{r_{\beta_i,n}}^{1-\frac{|V_i|}{2}}M.
$$
If $\pi$ contains a block of size greater than $2$, then the limit is $0$, and since the number of partitions of $d$ elements is finite, the statement follows.
\end{proof}

By \Cref{lem:appear_twice}, in the limit \eqref{eq:EtrH}, it suffices to consider the case where the number of factors is even. Assume that $\epsilon:[2d]\rightarrow \mathcal{I}$. Then we have 
$$\begin{aligned}
&\lim_{n\rightarrow \infty} 
\mathbb{E}\Bigl[\tr\bigl(H_{\epsilon(1),n} \cdots H_{\epsilon(2d),n}\bigr)\Bigr]\\
=&\lim\limits_{n\rightarrow \infty}\frac{(\sqrt{-1})^{\sum_{k=1}^{2d}\lfloor r_{\epsilon(k),n}/2 \rfloor}}
{\prod_{k=1}^{2d}\binom{n}{r_{\epsilon(k),n}}^{1/2}}
\sum_{\substack{R_k \in I_{\epsilon(k),n} ,\forall k=1,...,2d\\
 \ker [(\epsilon(1),R_{1}),...,(\epsilon(2d),R_{2d})] \in \pp_2(2d)}}
\mathbb{E}\!\left[\prod_{k=1}^{2d} J_{\epsilon(k),R_{k}}\right]\,
\tr\!\bigl(\Psi_{R_{1}}\cdots \Psi_{R_{2d}}\bigr)
\end{aligned}.$$

Similarly to the proof of \Cref{lem:appear_twice}, the condition
$$\ker [(\epsilon(1),R_{1}),...,(\epsilon(2d),R_{2d})] \in \pp_2(2d)$$
can be replaced by
$$\ker [R_{1},...,R_{2d}] \in \pp_2(2d) \quad\text{and}\quad \ker [R_{1},...,R_{2d}] \leq \ker\epsilon.$$

On the other hand, we have $\mathbb{E}\!\left[\prod_{k=1}^{2d} J_{\epsilon(k),R_{k}}\right]=1$ whenever $\ker [(\epsilon(1),R_{1}),...,(\epsilon(2d),R_{2d})] \in \pp_2(2d)$.
Therefore, we have 

$$\lim_{n\rightarrow \infty} 
\mathbb{E}\Bigl[\tr\bigl(H_{\epsilon(1),n} \cdots H_{\epsilon(2d),n}\bigr)\Bigr]
=\lim\limits_{n\rightarrow \infty}\frac{(\sqrt{-1})^{\sum_{k=1}^{2d}\lfloor r_{\epsilon(k),n}/2 \rfloor}}
{\prod_{k=1}^{2d}\binom{n}{r_{\epsilon(k),n}}^{1/2}}
\sum\limits_{\substack{\pi\in \pp_2(2d)\\ \pi\leq \ker\epsilon}}
\sum_{\substack{R_k \in I_{\epsilon(k),n} ,\forall k=1,...,2d\\
\ker [R_{1},...,R_{2d}]  =\pi}}
\tr\!\bigl(\Psi_{R_{1}}\cdots \Psi_{R_{2d}}\bigr).$$

Thus, it suffices to show that for all $\pi\in \pp_2(2d)$ such that $\pi\leq \ker\epsilon$, we have
\begin{equation}\label{eq:EtrH1}
\lim\limits_{n\rightarrow \infty}\frac{(\sqrt{-1})^{\sum_{k=1}^{2d}\lfloor r_{\epsilon(k),n}/2 \rfloor}}
{\prod_{k=1}^{2d}\binom{n}{r_{\epsilon(k),n}}^{1/2}}
\sum_{\substack{R_k \in I_{\epsilon(k),n} ,\forall k=1,...,2d\\
\ker [R_{1},...,R_{2d}]  =\pi}}
\tr\!\bigl(\Psi_{R_{1}}\cdots \Psi_{R_{2d}}\bigr)
= 
\prod_{i,j\in\mathcal I} q_{i,j}^{\,cr(\pi,\epsilon; i,j)}.
\end{equation}

By applying relations \eqref{eq:psi_selfadjoint} and \eqref{eq:psi_commute}, and by adjusting crossing pairs to noncrossing pairs, we obtain the following lemma.

\begin{lemma}\label{lem: probabilistic}
Let $\pi=\{(e_i,z_i)_{i=1}^{d}|1=e_1<e_2\dots<e_{d}<2d, e_i<z_i\} \in \pp_2(2d)$
and let $$E=\{(i,j)\mid (e_i,z_i) \text{ and } (e_j,z_j)\, \text{are crossing}\}.$$
If $\ker(R_1,...,R_{2d})=\pi$, then we have 
$$
(\sqrt{-1})^{\sum\limits_{k=1}^{2d}\lfloor r_{\epsilon(k),n}/2 \rfloor}\tr\!\bigl(\Psi_{R_{1}}\cdots \Psi_{R_{2d}}\bigr)=(-1)^{\sum\limits_{(i,j)\in E}\left|R_{\epsilon(e_i)}\cap R_{\epsilon(e_j)}\right|+ |R_{\epsilon(e_i)}|\cdot|R_{\epsilon(e_j)}|},
$$
where $|R_{e_i} \cap R_{e_j}|$ is the cardinality of the set $R_{e_i} \cap R_{e_j}$.
\end{lemma}

Notice that for $\pi\in \pp_2(2d)$ with $\pi\leq \ker\epsilon$, we have $\epsilon(e_i)=\epsilon(z_i)$.

Let $R_{\epsilon(e_i)}$ be a uniformly distributed random set of size $r_{\epsilon(e_i),n}$ drawn from $A_{\epsilon(e_i),n}$. By the proof of \Cref{lem:appear_twice}, the probability that $R_{\epsilon(e_i)}=R_{\epsilon(e_j)}$ tends to $0$.
Together with \Cref{lem: probabilistic}, \eqref{eq:EtrH1} is equivalent to

\begin{equation}\label{eq: EtrH2}
\lim_{n \to \infty} \E\left[ (-1)^{\sum_{(i,j)\in E} |R_{\epsilon(e_i)} \cap R_{\epsilon(e_j)}|+ |R_{\epsilon(e_i)}|\cdot|R_{\epsilon(e_j)}|} \right]=\prod_{i,j \in \mathcal{I}} q_{i,j}^{cr(\pi,\epsilon; i,j)}.
\end{equation}

Recall that $$cr(\pi,\epsilon; i,j)= \#\left\{ (v_1, v_2) \in \pi \times \pi \mid v_1 \subset \epsilon^{-1}(i), v_2\subset \epsilon^{-1}(j), v_1<v_2 \text{ and } v_1, v_2 \text{ are crossing} \right\},$$
only  $q_{i,j}$ with $i,j\in \epsilon([2d])$ contribute.
$$
\begin{aligned}
&cr(\pi,\epsilon; \epsilon(e_i),\epsilon(e_j))\\ =& \#\left\{ (v_1, v_2) \in \pi \times \pi \mid v_1 \subset \varepsilon^{-1}(\epsilon(e_i)), v_2\subset \varepsilon^{-1}(\epsilon(e_j)), v_1<v_2 \text{ and } v_1, v_2 \text{ are crossing} \right\}\\
=& \#\left\{ (i',j') \in E\mid \epsilon(e_{{i'}})=\epsilon(e_i), \epsilon(e_{j'})=\epsilon(e_j)     \right\}
\end{aligned}
$$

It follows that
$$ q_{\epsilon(e_i),\epsilon(e_j)}^{cr(\pi,\epsilon; \epsilon(e_i),\epsilon(e_j))}=\prod\limits_{\substack{(i',j')\in E\\
\epsilon(e_{i'})=\epsilon(e_i), \epsilon(e_{j'})=\epsilon(e_j)}}q_{\epsilon(e_{i'}),\epsilon(e_{j'})}.$$

Therefore, we have 
$$\prod_{ i,j\in \mathcal{I}} q_{i,j}^{cr(\pi,\epsilon; i,j)}=\prod\limits_{(i,j)\in E
}q_{\epsilon(e_{i}),\epsilon(e_{j})}.$$

Notice that $(-1)^{|R_{\epsilon(e_i)}|\cdot|R_{\epsilon(e_j)}|}=(-1)^{r_{\epsilon(e_i)} r_{\epsilon(e_j)}}$ and 
$q_{i,j}=(-1)^{r_i r_j}e^{-2\lambda_{i,j}}$.
Thus, \eqref{eq: EtrH2} is equivalent to 

\begin{equation}\label{eq: EtrH3}
\lim_{n \to \infty} \mathbb{E}\left[ (-1)^{\sum_{(i,j)\in E} |R_{\epsilon(e_i)} \cap R_{\epsilon(e_i)}|+ r_{\epsilon(e_i)} r_{\epsilon(e_j)}} \right]=\prod_{(i,j) \in E} (-1)^{r_{\epsilon(e_i)} r_{\epsilon(e_j)}} e^{-2\lambda_{\epsilon(e_{i}),\epsilon(e_{j})}}.
\end{equation}

Therefore, it suffices to show that 

\begin{equation}\label{eq: EtrH4}
\lim_{n \to \infty} \mathbb{E}\left[ (-1)^{\sum_{(i,j)\in E} |R_{\epsilon(e_i)} \cap R_{\epsilon(e_j)}|} \right]=\prod_{(i,j) \in E} e^{-2\lambda_{\epsilon(e_{i}),\epsilon(e_{j})}}.
\end{equation}

Recall the assumption that $\{R_{\epsilon(e_i)}\mid i=1,...,d\}$ are independent. In summary, to prove \Cref{main theorem}, it suffices to prove the following proposition.

\begin{proposition}\label{Main Proposition}

Let $d\in \N$ and  $[d]=\{1,...,d\}$.  For each $k\in [d]$ and $n\in \N$, let $A_{k,n}\subset \N$ such that $|A_{k,n}|=n$ and let $r_{k,n}\in \N$ such that $\lim\limits_{n\rightarrow\infty} \frac{r_{k,n}}{n}=0.$  For all $1\leq i<j\leq d$ 
$$
\lim\limits_{n\rightarrow \infty}\frac{r_{i,n}r_{j,n}}{n} \cdot \frac{\lvert A_{i,n}\cap A_{j,n} \rvert}{n}=\lambda_{i,j} \in[0, \infty].
$$
Let $E\subset\{(i,j)\mid 1\leq i<j\leq d\}$, and for each $i\in [d]$, let $R_{i}$ be a random ordered sequence of length $r_{i,n}$ uniformly chosen from $A_{i,n}.$
Then 
$$\lim_{n \to \infty} \mathbb{E}\left[ (-1)^{\sum_{(i,j)\in E} |R_{i} \cap R_{j}|} \right]=\prod_{(i,j) \in E} e^{-2\lambda_{i,j}},$$
with  the convention $e^{-\infty}=0.$
\end{proposition}

In the above proposition, $R_{i}$ plays the role of $R_{\epsilon(e_i)}$ in \eqref{eq: EtrH4}. 
The proof of \Cref{Main Proposition} splits into the following two cases. 
Since we will analyze the expectation for fixed n first and then take the limit, we omit the subindex $n$ below when there is no risk of confusion.

Let $a_{ij,n} = |A_{i,n} \cap A_{j,n}|$ denote the overlap size between domains.

\begin{lemma} \label{lem:singular}
    Let $(R_{1},\dots,R_{d})$ be independent random sets chosen uniformly from $A_{1,n},\cdots, A_{d,n}$.  If there exists $(i,j)\in E$, such that 
$\lambda_{i,j}=\lim\limits_{n\to \infty}\frac{r_{i,n} r_{j,n} a_{ij,n}}{n^2}=\infty $, then the expectation of the left term of \eqref{eq: EtrH4} vanishes, yielding the desired result:
    $$
    \lim_{n \to \infty} \mathbb{E}\left[ (-1)^{\sum_{(i,j)\in E} |R_{i} \cap R_{j}|} \right] = 0 =\prod_{(i,j) \in E} e^{-2\lambda_{i,j}}.
    $$
\end{lemma}

The second  case is that for all $(i,j)\in E$
$$\lambda_{i,j}=\lim\limits_{n\to \infty}\frac{r_{i,n} r_{j,n} a_{ij,n}}{n^2}<\infty .$$ 
For $(i,j)\in E$, let $X_{i,j}=R_{i}\cap R_{j}$ denote the intersection of the two random subsets $R_{i}$ and $R_{j}$ chosen uniformly from the base sets $A_{i,n}$ and $A_{j,n}$, respectively.
Then $|X_{i,j}|$, the cardinality of the random set $X_{i,j}$, is a random variable. 
Denote the falling factorial by
$$(x)_k:=x(x-1)\cdots(x-k+1).$$  
Then we have  the following limit.

\begin{lemma}\label{lem: n-poisson}
If $\lambda_{i,j}=\lim\limits_{n\to \infty}\frac{r_{i,n} r_{j,n} a_{ij,n}}{n^2}<\infty $ for all $(i,j)\in E$, then, for any $k\in\mathbb{N}$, we have 
    $$
     \lim\limits_{n\to \infty}  \E\left[\binom{\sum_{(i,j)\in E}|X_{i,j}|}{k}\right]= \lim\limits_{n\to \infty}  \E\left[\frac{1}{k!}\left(\sum_{(i,j)\in E}|X_{i,j}|\right)_{k}\right]=\sum_{\substack{k_{i,j}\in \N\\
      \sum_{(i,j)\in E} k_{i,j}=k}}\prod_{(i,j)\in E} \frac{\lambda_{i,j}^{k_{i,j}}}{k_{i,j}!}.
    $$
\end{lemma}

Notice that, for $X=\sum_{(i,j)\in E}|X_{i,j}|,$
$$
    (-1)^{|X|}=(1-2)^{|X|}=\sum_{k=0}^{\infty}\binom{|X|}{k}(-2)^{k}=\sum_{k=0}^{\infty}\frac{[(|X|)_k]}{k!}(-2)^{k},
$$ 
by \Cref{lem: n-poisson}, we have
$$
\begin{aligned}
 \lim_{n \to \infty} \mathbb{E}\left[ (-1)^{\sum_{(i,j)\in E} |R_{i} \cap R_{j}|} \right]
    &=\lim_{n \to \infty}\E[(-1)^{\sum_{(i,j)\in E}|X_{i,j}|}]\\
    &=\lim_{n \to \infty}\sum_{k=0}^{\infty}\E[\frac{[(\sum_{(i,j)\in E}|X_{i,j}|)_k]}{k!}(-2)^{k}].\\
     \end{aligned}
$$   

It remains to prove the following equality to complete the proof of \Cref{Main Proposition}.

$$
\begin{aligned}   
\lim_{n \to \infty}\sum_{k=0}^{\infty}\E[\frac{(\sum_{(i,j)\in E}|X_{i,j}|)_k}{k!}(-2)^{k}]
    =\prod_{(i,j) \in E} e^{-2\lambda_{i,j}}.
    \end{aligned}
$$
The proof is given in the next section.    

We provide an example in which $\lim\limits_{n\rightarrow \infty} \frac{r_{k,n}}{n}\neq 0$ for some $k$, and for which the limit in \Cref{Main Proposition} fails.
\begin{example}\label{counter-examples}
Let $n$ be an even number, $A_{1,n}=\{1,...,n\}$ and $A_{2,n}=\{n,...,2n-1\}$, $r_{1,n}=r_{2,n}=n/2$. Then, $|A_{1,n}\cap A_{2,n}|=1$ and 
$$\lim\limits_{n\rightarrow \infty}\frac{r_{1,n}r_{2,n}}{n} \cdot \frac{\lvert A_{1,n}\cap A_{2,n} \rvert}{n}=\frac{1}{4}.$$
For $k\geq 2$, $\mathbb{P}(|R_1\cap R_2|=k)=0$ and 
$$\mathbb{P}(|R_1\cap R_2|=1)=\frac{\binom{n-1}{n/2-1}\binom{n-1}{n/2-1}}{\binom{n}{n/2}\binom{n}{n/2}}=\frac{1}{4}.$$
It follows that
$$\lim_{n \to \infty} \mathbb{E}\left[ (-1)^{|R_{1} \cap R_{2}|} \right]
=(1-\frac{1}{4})-\frac{1}{4}=\frac{1}{2}\neq e^{-2\cdot\frac{1}{4}}=e^{-\frac{1}{2}}.$$
\end{example}

\section{ Proofs of technical lemmas  and comments}

\subsection{Proof of \Cref{lem:singular}}
\begin{proof}[Proof of \Cref{lem:singular}]

The proof of this lemma is similar to that in~\cite{FTW19}; the main difference is that we must account for overlaps between subsystems.

First, note that $R_{1},\dots,R_{d}$ are mutually independent.
Without loss of generality, assume $(u,v)=(1,2)\in E$, that is $\lambda_{1,2}=\lim\limits_{n\to \infty}\frac{r_{1,n}r_{2,n}a_{12,n}}{n^{2}}=\infty$. 
Let $S=\bigcup_{k=3}^{d}R_{k}$ be the set occupied by all the other terms. 
Decompose $R_{1}$ and $R_{2}$ into  fixed and free parts relative to $S$:
\begin{itemize}
    \item Fixed parts: $T_{1} = R_{1} \cap S$ and $T_{2} = R_{2} \cap S$.
    \item Free parts: $V_{1} = R_{1} \setminus S$ and $V_{2} = R_{2} \setminus S$.
\end{itemize}

 Let $\mathcal{F}$ be the $\sigma$-algebra generated by $\{R_{k}\mid k=3,...,d\}$ and the intersections $T_{1}$ and $T_{2}$.
Conditioned on $\mathcal{F}$, the sets $V_{1}$ and $V_{2}$ are independent. Specifically:
\begin{itemize}
    \item $V_{1}$ is uniformly distributed over $A_{1} \setminus S$ with size $|V_{1}| = r_{1,n} - |T_{1}|$.
     \item $V_{2}$ is uniformly distributed over $A_{2} \setminus S
     $ with size $|V_{2}| = r_{2,n} - |T_{2}|$.
\end{itemize}

Conditioning on $\mathcal{F}$, for any $(u,v)\neq (1,2)$,  $\lvert R_{u}\cap R_{v}\rvert$ is measurable with respect to $\mathcal{F}$ and $R_{1} \cap R_{2}$ can be decomposed into the random term $V_{1}\cap V_{2}$ and the fixed term $T_{1}\cap T_{2}$ relative to $S$.
Therefore, we have 
\begin{align*}
\left|\E\left[ (-1)^{\sum\limits_{(i,j)\in E} \lvert R_{i} \cap R_{j}\rvert} \right]\right|
&= \E\left[ (-1)^{\sum\limits_{(1,2)\neq(i,j)\in E} \lvert R_{i} \cap R_{j}\rvert} \E\left[(-1)^{\lvert R_{1} \cap R_{2}\rvert}\mid \mathcal{F}\right]\right] \\
&\leq  \E\left[\left\lvert \E\left[ (-1)^{\lvert R_{1} \cap R_{2}\rvert} \mid \mathcal{F}\right] \right\rvert \right]\\
&=\E \left[\left\lvert \E\left[ (-1)^{\lvert V_{1} \cap V_{2}\rvert} \mid \mathcal{F} \right] \right\rvert \right].
\end{align*}

The last equality follows from the fact that the random part of $R_{1}\cap R_{2}$ given $\mathcal{F}$ lies outside $S$.
It suffices to estimate the random variable 
$$
\lambda=\left\lvert \E\left[ (-1)^{\lvert V_{1} \cap V_{2}\rvert} \mid \mathcal{F} \right]\right\rvert.
$$

Unlike the fully overlapping case, $V_{1}$ and $V_{2}$ lie in different domains; their intersection can occur only in  $A_{1}\cap A_{2}$. 
We perform a further conditioning step to derive the combinatorial formula.

Fix $V_{1}$ and treat $V_{2}$ as the only random variable.
Let $U_{2} = A_{2} \setminus S$ be the available domain for $V_{2}$. 
Then $\lambda$ depends on the part of $V_{1}$ that falls into the domain of $V_{2}$.
Define 
$$
\tilde{V}_{1} = V_{1} \cap U_{2} = V_{1} \cap A_{2}.
$$

Conditioned on $\mathcal{F}$ and $V_{1}$,  $|V_{1} \cap V_{2}|=|\tilde{V}_{1} \cap V_{2}|$.
Hence, we have the conditional expectation:
$$\mathbb{E}\left[ (-1)^{|V_{1} \cap V_{2}|} \mid \mathcal{F}, V_{1} \right] = \sum_{k=0}^{\infty} (-1)^k \frac{\binom{\lvert \tilde{V}_{1}\rvert }{k}\binom{\lvert U_{2} \rvert -\lvert \tilde{V}_{1} \rvert}{\lvert V_{2}\rvert -k}}{\binom{\lvert U_{2} \rvert }{\lvert V_{2}\rvert}}.$$
Define
$$ F(p, q, m)=\sum_k(-1)^k\binom{p}{k}\binom{m-p}{q-k} /\binom{m}{q} .$$
Then
$$ \mathbb{E}\left[ (-1)^{|V_{1} \cap V_{2}|} \mid \mathcal{F}, V_{1} \right] =F(\lvert \tilde{V}_{1} \rvert, \lvert V_{2} \rvert, \lvert U_{2} \rvert).$$

By the proof of Lemma 5 in~\cite{FTW19}, we have 

\begin{equation}\label{eq:estimation for F}
	F(p,q,m)\leq e^{- a_p a_q/2m},
\end{equation}
where $a_p=\min\{p,m-p\}$.

In the following, we denote by $m = \lvert U_{2} \rvert, p = \lvert \tilde{V}_{1} \rvert$ and $q = \lvert V_{2} \rvert.$
For any $k\in \{3,\dots,d\}$ and each $x\in\N$, we have
$$
\mathbb{P}(x\notin R_{k})=
\begin{cases}
1-\dfrac{r_{k,n}}{n}\ge \dfrac12, & x\in A_{k}\\[0.6em]
1, & x\notin A_{k}.
\end{cases}
$$

By independence and the sampling rule, we have $|R_{k}|=r_{k,n}\le n/2$,
hence the survival probability $\mathbb{P}(x\notin S)$ has the uniform lower bound:

\begin{equation}
\label{eq:rho_lower}
\mathbb{P}(x\notin S)=\prod_{k=3}^{d}\mathbb{P}(x\notin R_{k}) \ge 2^{-(d-2)}>0.
\end{equation}
Let $$c_d=2^{-(d-2)}, \quad \Lambda_{2}:=\sum_{x\in A_{2,n}} \mathbb{P}(x\notin S),\quad
\Lambda_{12}:=\sum_{x\in A_{1,n}\cap A_{2,n}}\mathbb{P}(x\notin S).$$
By \eqref{eq:rho_lower} and $|A_{2,n}|=n$, we have 
\begin{equation}\label{property of Lambda}
\Lambda_{2}\ge c_d\,n,\qquad
\Lambda_{12}\ge c_d\,a_{12,n}.
\end{equation}
The above lower bounds supply intuitive information, the fact the set still contains the same magnitude of elements as before. 
We also have the upper bound of $m$,
\begin{equation} \label{property of m_n}
	m= \lvert A_{2}\setminus S\rvert\leq \lvert A_{2}\rvert=n,
\end{equation}
which, together with the lower bounds, will be used to complete the estimate and finish the proof.

\medskip
All that remains is to bound $a_{p}$ and $a_{q}$ in probability to control $\lambda$, for which we first estimate the expectation and variance of $p$ and $m-p$ first.
\subsubsection{Estimate for $a_{p}$}
Recall that $p=\lvert \tilde{V}_{1} \rvert$. Its expectation is 

$$\begin{aligned}
	\E[p]&=\sum_{x\in A_{1}\cap   A_{2}} \Prob (x\in R_{1})\prod_{k=3}^{d}\Prob (x\notin R_{k})\\&
	 =\sum_{x\in A_{2}}\Prob(x\in R_{1})\prod_{k=3}^{d}\Prob (x\notin R_{k})\\&
	 =\frac{r_{1,n}}{n}\Lambda_{12}.
\end{aligned}$$

On the other hand, we have 

\begin{align*}
	\E[p^{2}]&=\sum_{x\in A_{1}\cap   A_{2}} \Prob (x\in \tilde{V}_{1})+\sum_{x\neq y \in A_{1}\cap   A_{2}}\Prob (x,y \in \tilde{V}_{1}) \\
	&=\E[p]+ \E[p]^2- \sum_{x\in A_{1}\cap   A_{2}} \Prob^2 (x\in \tilde{V}_{1})\\
	&<\E[p]+ (\E[p])^2.
\end{align*}

Thus, the variance satisfies:
$$
	\Var(p)=\E[p^{2}]-\left( \E[p]\right)^{2}<\E[p].
$$
By Chebyshev's inequality, we have 
\begin{equation}\label{eq:estimate for p_n}
	\Prob(p\leq \frac{1}{2}\E[p])\leq \Prob(\lvert p- \E[p]\rvert \ge \frac{1}{2}\E[p])\leq \frac{\Var(p)}{(\frac{1}{2}\E[p])^{2}}<\frac{4}{\E[p]}=\frac{4n}{r_{1,n}\Lambda_{12}}.
\end{equation}

Similarly, since
$$\E[m - p]=\Lambda_{2}-\frac{r_{1,n}}{n} \Lambda_{12}\geq\frac{1}{2}\Lambda_{2},$$
 we have 
\begin{equation}\label{eq:estimate for m_n-p_n}
	\Prob\left(m-p\leq \frac{1}{2}\E[m-p]\right)<\frac{4}{\E[m-p]}=\frac{8}{\Lambda_{2}}.
\end{equation}
Combining \eqref{eq:estimate for p_n} and \eqref{eq:estimate for m_n-p_n} and using $\E[p]\leq\E[m-p]$, we have the probabilistic bound for $a_{p}$:
\begin{equation}\label{probability bound for a_{p}}
	\Prob\left(a_{p}\leq \frac{1}{2}\E[p]\right)\leq \Prob\left(m-p\leq \frac{1}{2}\E[m-p]\right)+\Prob(p\leq \frac{1}{2}\E[p])\leq \frac{4n}{r_{1,n}\Lambda_{12}}+\frac{8}{\Lambda_{2}}.
\end{equation}

\subsubsection{Estimate for $a_{q}$}	
To compare $p$ and $q$, note that $p$ counts  elements in
$V_1 \cap A_{2}$ that are not in $S$, while $q$ counts elements in $A_{2}$ that are not in $S$, this corresponds to the difference between
$\Lambda_{2}$ and $\Lambda_{12}$. 
Hence, we have the following estimates, analogous to those for $a_{p}$:
\begin{enumerate}
	\item $\E[q]=\frac{r_{2,n}}{n}\Lambda_{2}$, $\Var(q)\leq \E[q]$.
	\item $\E[m-q]=(1-\frac{r_{2,n}}{n})\Lambda_{2}$, $\Var(m-q)\leq \E[m-q]$.
	\item $\E[q]\leq\E[m-q]$.
\end{enumerate}
The corresponding probabilistic bound for $a_{q}$ is
\begin{equation}\label{probability bound for a_{q}}
	\Prob\left(a_{q}\leq \frac{1}{2}\E[q]\right)\leq \Prob\left(m-q\leq \frac{1}{2}\E[m-q]\right)+\Prob(q\leq \frac{1}{2}\E[q])\leq \frac{4n}{r_{2,n}\Lambda_{2}}+\frac{8}{\Lambda_{2}}.
\end{equation}

Combining \eqref{probability bound for a_{p}} and \eqref{probability bound for a_{q}}, we have
$$
\begin{aligned}
	\Prob\left( \left\{ a_{p} \leq \frac{1}{2}\frac{r_{1,n}}{n}\Lambda_{12} \right\} \cup \left\{ a_{q} \leq \frac{1}{2}\frac{r_{2,n}}{n} \Lambda_{2} \right\} \right) 
		&\leq \Prob \left( \left\{ a_{p} \leq \frac{1}{2}\frac{r_{1,n}}{n}\Lambda_{12}\right\} \right)+\Prob\left( \left\{a_{q}\leq \frac{1}{2}\frac{r_{2,n}}{n}\Lambda_{2}\right\}\right)\notag\\
		&\leq \frac{4n}{r_{1,n}\Lambda_{12}}+\frac{8}{\Lambda_{2}}+\frac{4n}{r_{2,n}\Lambda_{2}}+\frac{8}{\Lambda_{2}}. 
\end{aligned}
$$

Now, applying the above estimate above to the bound \eqref{eq:estimation for F}, we have
$$
\begin{aligned}
	&\left|\E\left[ (-1)^{\sum_{(i,j)\in E} |R_{i} \cap R_{j}|} \right] \right|\\  
	\leq &  \E[ \lambda]\\
	\leq & \E[e^{-\frac{a_{p}a_{q}}{2m}}] \\
	< &\exp\left(-\frac{1}{8m}\frac{r_{1,n}r_{2,n}}{n}\frac{\Lambda_{12}\Lambda_{2}}{n}\right)+1\cdot \left(\frac{4n}{r_{1,n}\Lambda_{12}}+\frac{8}{\Lambda_{2}}+\frac{4n}{r_{2,n}\Lambda_{2}}+\frac{8}{\Lambda_{2}} \right)\\
	\leq &\exp \left(-\frac{c_{d}^{2}}{8}\frac{r_{1,n}r_{2,n}a_{12,n}}{n^{2}}\right)+\frac{4}{c_d}\left( \frac{n}{r_{1,n}a_{12,n}}+\frac{1}{r_{2,n}}+\frac{4}{n} \right).
\end{aligned}
$$
Since $ \lim\limits_{n\rightarrow \infty}\frac{r_{1,n}r_{2,n}a_{12,n}}{n^{2}}= \infty$, $\lim\limits_{n\rightarrow \infty}\frac{r_{1,n}}{n}=\lim\limits_{n\rightarrow \infty}\frac{r_{2,n}}{n}=0$ and $a_{12,n}<n$, we have
$$ \lim\limits_{n\rightarrow \infty}\frac{r_{1,n}r_{2,n}}{n}= \infty,\,\text{and}\,\lim\limits_{n\rightarrow \infty}\frac{r_{1,n}a_{12,n}}{n}= \infty, $$
thus,
$$\lim\limits_{n\rightarrow \infty} \frac{n}{r_{1,n}a_{12,n}}=0\,\text{and}\, \lim\limits_{n\rightarrow \infty} \frac{n}{r_{2,n}}=0. $$
It follows that 
$$\lim\limits_{n\rightarrow \infty} \E\left[ (-1)^{\sum_{(i,j)\in E} |R_{i} \cap R_{j}|} \right]=0.$$
\end{proof}

\subsection{Proof of \Cref{lem: n-poisson}}
Recall the assumption that for all $(i,j)\in E$
 $$\lambda_{i,j}=\lim\limits_{n\to \infty}\frac{r_{i,n} r_{j,n} a_{ij,n}}{n^2}<\infty .$$ 
 
\begin{lemma}\label{lem:pair_disjoint1}
For any two distinct pairs $(i_1,j_1)$ and $(i_2,j_2)\in E$, we have
$$ 
\lim\limits_{n\rightarrow \infty}\Prob\big( X_{i_1,j_1}\cap X_{i_2, j_2}\neq \emptyset \big)=0.$$
\end{lemma}

\begin{proof}
Without loss of generality, assume $a_{i_1 j_1,n} = |A_{i_1,n} \cap A_{j_1,n}| \geq a_{i_2 j_2,n} = |A_{i_2,n} \cap A_{j_2,n}|$. 
Notice that 
$$X_{i_1,j_1}\cap X_{i_2, j_2}\subset A_{i_1,n} \cap A_{j_1,n}\cap A_{i_2,n} \cap A_{j_2,n}. $$
Let $S=A_{i_1,n} \cap A_{j_1,n}\cap A_{i_2,n} \cap A_{j_2,n}$. 

Case 1: The two intersections involve exactly three distinct base sets. Without loss of generality, we may assume that $j_1=j_2$. 
Then we have

$$
\begin{aligned}
&\Prob\big( X_{i_1,j_1}\cap X_{i_2, j_2}\neq \emptyset \big)\\
\leq& \sum_{b \in S} \Prob\big(b \in R_{i_1,n} \cap R_{j_1,n} \cap R_{i_2,n}
\big) \\
=&\sum_{b \in S} \frac{\binom{n-1}{r_{i_1,n}-1}\binom{n-1}{r_{j_1,n}-1}\binom{n-1}{r_{i_2,n}-1}}{\binom{n}{r_{i_1,n}}\binom{n}{r_{j_1,n}}\binom{n}{r_{i_2,n}}}  \\
=&\sum_{b \in S} \frac{r_{i_1,n}}{n} \cdot \frac{r_{j_1,n}}{n} \cdot \frac{r_{i_2,n}}{n} .
\end{aligned}
$$
Notice that 
$|S|\leq \min \{a_{i_1 j_1,n},a_{i_2 j_2,n} \}$. 
It follows that 
 $$\Prob\big( X_{i_1,j_1}\cap X_{i_2, j_2}\neq \emptyset \big)
\leq a_{i_1 j_1,n}\frac{r_{i_1,n}}{n} \cdot \frac{r_{j_1,n}}{n} \cdot \frac{r_{i_2,n}}{n}. $$
Since  $ \lim\limits_{n\to \infty} \frac{r_{i_2,n} }{n}=0$, we have 

$$\lim\limits_{n\rightarrow\infty} \Prob\big( X_{i_1,j_1}\cap X_{i_2, j_2}\neq \emptyset \big)\leq \lim\limits_{n\rightarrow\infty}\frac{r_{i_2,n}}{n} \cdot \frac{a_{i_1 j_1,n} r_{j_1,n}r_{i_1,n}}{n^2}=0\cdot\lambda_{i_1,j_1}=0. $$
  
Case 2: $i_1, j_1, i_2, j_2$ are all distinct.
 Then we have 
 
$$
\begin{aligned}
\Prob\big( X_{i_1,j_1}\cap X_{i_2, j_2}\neq \emptyset \big)
\leq& \sum_{b \in S} \Prob\big(b \in R_{i_1,n} \cap R_{j_1,n} \cap R_{i_2,n} \cap R_{j_2,n
}\big) \\
\leq& a_{i_1 j_1,n}\frac{r_{i_1,n}}{n} \cdot \frac{r_{j_1,n}}{n} \cdot \frac{r_{i_2,n}}{n} \cdot \frac{r_{j_2,n}}{n}\to 0.
\end{aligned}
$$
\end{proof}

Now we are ready to prove \Cref{lem: n-poisson}.

\begin{proof}[proof of \Cref{lem: n-poisson}]
Let $A_{ij,n}=A_{i,n}\cap A_{j,n}.$
For any $k\in\mathbb{N}$, by \Cref{lem:pair_disjoint1}, we may assume that the sets $X_{i,j}$'s are pairwise disjoint. Then we have 
    $$
     \E\left[ \binom{\sum_{(i,j)\in E}|X_{i,j}|}{k}\right]= \sum \limits_{|B|=k }
\mathbb{P}\big(B\subseteq \bigcup\limits_{(i,j)\in E}X_{i,j}\big)+o(1).
$$

Since the $X_{i,j}$'s are pairwise disjoint, each $B\subseteq \bigcup\limits_{(i,j)\in E}X_{i,j}$ has a unique decomposition  $B=\bigsqcup\limits_{(i,j)\in E}B_{i,j}$ with $B_{i,j}\subset X_{i,j}$.  Conversely, any choice of subsets $B_{i,j}\subset X_{i,j}$ with $\sum \limits_{(i,j)\in E}|B_{i,j}|=k$ yields
$B=\bigcup\limits_{(i,j)\in E}B_{i,j}$ of size $k$. 
Thus, we have 
$$ \begin{aligned}
\mathbb{P}\big(B\subseteq \bigcup\limits_{(i,j)\in E}X_{i,j}\big)&=\sum\limits_{\sum \limits_{(i,j)\in E}|B_{i,j}|=k }\Prob\left( \bigcap_{(i,j) \in E} \{B_{i,j} \subset X_{i,j}\} \right)\\&= \sum\limits_{\substack{|B_{i,j}|=k_{i,j}\\\sum \limits_{(i,j)\in E}|B_{i,j}|=k }}\Prob\left( \bigcap_{(i,j) \in E} \{B_{i,j} \subset X_{i,j}\} \right).
\end{aligned}$$

For a fixed family $\{B_{i,j}\mid |B_{i,j}|=k_{i,j}\}$, let $c_i = \sum_{(i,j) \in E} k_{i,j} + \sum_{(j,i) \in E} k_{j,i}$ denote the total number of distinct elements required from $R_i$.
 Since the $B_{i,j}$'s are disjoint, the required elements in each $R_i$ are distinct. By  independence of $\{R_i\}$, the probability factors as
$$
\Prob\left( \bigcap_{(i,j) \in E} \{B_{i,j} \subset X_{i,j}\} \right) = \prod_{i=1}^{d} \frac{(r_{i,n})_{c_i}}{(n)_{c_i}}.
$$
Let $D$ be the number of such disjoint families $\{B_{i,j}\mid \lvert B_{i,j}\rvert=k_{i,j}\}$. 
Since $B_{i,j}$ is chosen from $A_{ij,n}$, we have 
 $$\prod_{(i,j) \in E} \binom{a_{ij,n}-k}{k_{i,j}}\leq D\leq \prod_{(i,j) \in E} \binom{a_{ij,n}}{k_{i,j}}.$$
Hence
 $$\lim\limits_{n\rightarrow \infty} \frac{D}{\prod_{(i,j) \in E}a_{ij,n}^{k_{i,j}}}=\prod_{(i,j) \in E}\frac{1}{k_{i,j}!}.$$
 
The $c_i$'s are fixed finite numbers, thus 
 $$ \lim\limits_{n\rightarrow \infty}\frac{(r_{i,n})_{c_i}/(n)_{c_i}}{(r_{i,n}/n)^{c_i}}=1.$$
Regrouping the powers of $r_i/n$ from $\{R_i\}_{i=1}^{d}$ according to $(i,j)\in E$, we have 
$$
\prod_{i=1}^{d} \left( \frac{r_{i,n}}{n} \right)^{c_i} = \prod_{(i,j) \in E} \left( \frac{r_{i,n} r_{j,n}}{n^2} \right)^{k_{i,j}}.
$$

Therefore, 
$$
\begin{aligned}
&\lim\limits_{n\rightarrow \infty}\sum\limits_{|B_{i,j}|=k_{i,j}}\Prob\left( \bigcap_{(i,j) \in E} \{B_{i,j} \subset X_{i,j}\} \right)\\
=&\lim\limits_{n\rightarrow \infty}D \prod_{i=1}^{d} \frac{(r_{i,n})_{c_i}}{(n)_{c_i}}\\
=&\lim\limits_{n\rightarrow \infty}\prod_{(i,j) \in E} \frac{1}{k_{i,j}!}\left( a_{ij,n} \frac{r_{i,n} r_{j,n}}{n^2} \right)^{k_{i,j}} \\
=& \prod_{(i,j) \in E}\frac{\lambda_{i,j}^{k_{i,j}}}{k_{i,j}!}.
\end{aligned}
$$
The statement follows.
\end{proof}

\subsection{Proof of \Cref{Main Proposition}}

By \Cref{lem:singular}, the statement holds if there exists $(i,j)\in E$, such that $\lambda_{i,j}=\lim\limits_{n\to \infty}\frac{r_{i,n} r_{j,n} a_{ij,n}}{n^2}= \infty$.

Now, we assume that $\lambda_{i,j}=\lim\limits_{n\to \infty}\frac{r_{i,n} r_{j,n} a_{ij,n}}{n^2}<\infty $ for all $(i,j)\in E$. 

Then, there exists $M>0$ such that $|\frac{r_{i,n} r_{j,n} a_{ij,n}}{n^2}|\leq M$ for all $(i,j)\in E$ and $n\in \N$. 

Since $r_{i,n}<n$, we have $\frac{r_{i,n}}{n}\geq \frac{r_{i,n}-k}{n-k}$ for any $0\leq k< r_{i,n}$. Using the notation and computations from the proof of \Cref{lem: n-poisson},  we have that 

$$\sum\limits_{|B_{i,j}|=k_{i,j}}\Prob\left( \bigcap_{(i,j) \in E} \{B_{i,j} \subset X_{i,j}\} \right)<\prod_{(i,j) \in E}\frac{1}{k_{i,j}!}\left( a_{ij,n} \frac{r_{i,n} r_{j,n}}{n^2} \right)^{k_{i,j}}<\prod_{(i,j) \in E}\frac{1}{k_{i,j}!}\left(M \right)^{k_{i,j}}.$$
 
Thus, we have
$$
\E\left[ \binom{\sum_{(i,j)\in E}|X_{i,j}|}{k}\right]
\le
\sum_{\substack{k_{i,j}\in\mathbb{N}\\ \sum_{(i,j)\in E} k_{i,j}=k}}
\prod_{(i,j)\in E}\frac{M^{k_{i,j}}}{k_{i,j}!}.
$$

Let
$$
a_k=\sum_{\substack{k_{i,j}\in\mathbb{N}_0\\ \sum_{(i,j)\in E} k_{i,j}=k}}
\prod_{(i,j)\in E}\frac{M^{k_{i,j}}}{k_{i,j}!}.
$$

Then,
$$\sum\limits_{k=0}^\infty 2^k a_k= e^{2M|E|}<\infty.$$

By the dominated convergence theorem, we have

$$
\begin{aligned}   
\lim_{n \to \infty}\sum_{k=0}^{\infty}\E[\frac{[(\sum_{(i,j)\in E}|X_{i,j}|)_k]}{k!}(-2)^{k}]
    &=\sum_{k=0}^{\infty}\sum_{\substack{k_{i,j}\in \N\\
      \sum_{(i,j)\in E} k_{i,j}=k}}\prod_{(i,j)\in E} \frac{\lambda_{i,j}^{k_{i,j}}}{k_{i,j}!}(-2)^{k_{i,j}}\\
    &=\prod_{(i,j)\in E}\sum_{k_{i,j}=0}^{\infty} \frac{\lambda_{i,j}^{k_{i,j}}}{k_{i,j}!}
      (-2)^{k_{i,j}}\\
      &=\prod_{(i,j) \in E} e^{-2\lambda_{i,j}}.
    \end{aligned}
$$

This completes the proof of \Cref{Main Proposition}. \Cref{main theorem} then follows.

\subsection{Comments and further questions}

In summary, we have evaluated certain joint distributions of SYK models from different systems with overlaps. 
By Example \ref{counter-examples}, we know that \Cref{main theorem} may fail if $\lim\limits_{n\rightarrow \infty}\frac{r_{i,n}}{n}\neq 0$.
Therefore, we have not considered the case for $\lim\limits_{n\rightarrow \infty}\frac{r_{i,n}}{n}\neq 0$ in general. 
This excludes interesting examples such as $r_{1,n}=\frac{n}{r}$, $r_{2,n}=\lfloor\sqrt{sn}\rfloor$, and $a_{12,n}=\lfloor\sqrt{n}\rfloor$ for $r\geq 2$ and $s\geq 0$. 
In this case, we know that the first model converges to a semicircular element, and the second model is $q$-Gaussian with $q=e^{-2s}$. 
However, it is not easy to determine the $q_{i,j}$-relations between the limits of these two models.  

By adjusting the intersections of $A_{i,n}\cap A_{j,n}$ in \Cref{main theorem} , we are able to obtain all limiting distributions for mixed $q$-Gaussian systems for $q_{i,i}=0$ and $q_{i,j}\in [0,1]$. 
Notice that to obtain negative $q_{i,j}$, it requires $r_{i,n}$ and $r_{j,n}$ to be odd. It follows that, if $q_{i,j}$ and $q_{j,k}$ are negative, then $r_{i,n}$ and $r_{k,n}$ are odd, and $q_{i,k}$ must be non-positive in our SYK models.

Now we pose two questions related to our work:
\begin{enumerate}
    \item  In general, what are the $q_{i,j}$-relations between SYK models if there exists at least one index $i$ with $\lim_{n\to\infty}\frac{r_{i,n}}{n}\neq 0$.
    \item How can one obtain all mixed $q$-Gaussian relations from the limiting joint distributions of SYK models?
\end{enumerate}

\vspace{2cm}

\noindent{\bf Acknowledgment:} This work was supported by Zhejiang Provincial Natural Science Foundation of China under
Grant No. LR24A010002 and  National Natural Science Foundation of China under Grant No. 12171425.

\bibliographystyle{plain}

\bibliography{references}

\end{document}